\documentclass[]{amsart}

\usepackage{amsmath, amssymb,amsthm}
\usepackage{mathtools}

\usepackage{hyperref}
\usepackage{pgf,tikz}
\usepackage{amsaddr}

\usepackage{caption}

\newtheorem{lemma}{Lemma}[section]
\newtheorem{theorem}[lemma]{Theorem}
\newtheorem{proposition}[lemma]{Proposition}

\newtheorem*{thm*}{Theorem} 

\theoremstyle{definition}
\newtheorem{definition}[lemma]{Definition}
\newtheorem{remark}[lemma]{Remark}
\newtheorem{example}[lemma]{Example}

\newtheorem{construction}[lemma]{Construction}

\DeclareMathOperator{\End}{End}
\DeclareMathOperator{\Aut}{Aut}
\DeclareMathOperator{\ad}{ad}

\DeclareMathOperator{\Inn}{Inn}
\DeclareMathOperator{\Lie}{Lie}

\title{Hermitian cubic norm structures and groups of relative rank one}
\author{Michiel Smet}
\email{Michiel.Smet@UGent.be}

\begin{document}
	
	\maketitle
	
	\begin{abstract}
		Hermitian cubic norm structures were recently introduced in order to study the class of skew-dimension one structurable algebras (which are typically only defined over fields of characteristic different from $2$ and $3$) over arbitrary rings and fields.		
		Here, we generalize the quartic norm for these algebras and show that elements for which the quartic norm is invertible are conjugate invertible.
		
		This leads to the notion of a division hermitian cubic norm structure, defined as one in which the quartic norm is anisotropic. We classify such structures in terms of the Tits index of an associated (rank one) adjoint simple linear algebraic group and show that any adjoint linear algebraic group with such a Tits index defines a corresponding hermitian cubic norm structure.
	\end{abstract}
	
	\section{Introduction}
	
	Structurable algebras were introduced by Allison \cite{allison1978class, Allison1979ModelsOI} and were used to construct certain isotropic simple Lie algebras, with particular emphasis on the exceptional Lie algebras. Division structurable algebras are especially useful for understanding simple algebraic groups of rank one \cite{Tom}. However, these algebras are typically only considered over fields of characteristic different from $2$ and $3$, though there are occasional generalizations to broader settings \cite{AllisonHein81,Allison1993NonassociativeCA, OKP}.
	
	In the case of skew-dimension $1$ structurable algebras, which are quite well-understood \cite{Allison84, All90, Tom2019}, we use the hermitian cubic norm structures introduced in \cite{Michiel2025} to explore what happens over arbitrary rings. More precisely, a hermitian cubic norm structure defines a ``skew-dimension $1$" structurable algebra $A$ such that
	\begin{itemize}
		\item the construction of the Lie algebra employed in \cite{Allison1979ModelsOI} still works, and
		\item one has a nicely behaved ``projective elementary group", as studied in \cite{All99} for structurable algebras and Kantor pairs and, for example, in \cite{Faulkner89,Loos95,Loos19} for Jordan pairs.
	\end{itemize} Such a nicely behaved projective elementary group is precisely what we need if we want to study rank $1$ linear algebraic groups.
	
	First, we recover a generalization of the often recurring quartic norm on the structurable algebra \cite{Allison84, All90, Tom2019}. For this generalization we are able to show that an element is conjugate invertible in the terminology of \cite{Allison84} (or equivalently one-invertible in the terminology of \cite{All99}), if and only if its quartic norm is invertible. 
	
	Having established this, we are in a position to classify the division hermitian cubic norm structures. These fall into some distinct classes that are related to algebraic groups of rank one. To be precise, in the finite-dimensional case one has, depending on the map $\sharp$, either
	\begin{enumerate}
		\item $\sharp = 0$, in which case the associated group is of the form $\prescript{2}{}{A}_{n+2,1}$, or
		\item $\sharp \neq 0$, which implies that the hermitian cubic norm structure comes from a semisimple cubic norm structure $J$ over a field, in which case:
		\begin{itemize}
			\item $J$ is absolutely simple and the associated group is of type $\prescript{2}{}{E}^{35}_{6,1}$, $E^{66}_{7,1}$ or $E_{8,1}^{133}$, or
			\item $J$ is not absolutely simple, with associated group $\prescript{1,2}{}{D}_{n,1}$, $\prescript{3,6}{}{D}^{9}_{4,1}$.
		\end{itemize}
	\end{enumerate}
	Furthermore, any simple linear algebraic group with Tits index listed above comes from a hermitian cubic norm structure.
	Infinite dimensional division hermitian cubic norm structures do also exist, and these can be thought of as falling into the classes $\prescript{2}{}{A}_{n,1}$ and $\prescript{1,2}{}{D}_{n,1}$.

	Since Hermitian cubic norm structures, and their split counterparts cubic norm pairs, can be placed in the context of Tits polygons \cite{muhlherr2022}, we remark that the non-infinite families we mention here are basically the buildings that are looked at in \cite[Chapter 3]{muhlherr2022}.

	Finally, we study the associated Moufang set with our division structurable algebra, employing different constructions. Using the elementary group divided out by a parabolic, we give a first, direct construction. Afterwards, we give a construction more in line with \cite{Tom06, Tom}, i.e., using a single root group and a permutation of the Moufang set, or the division pair approach of \cite{Loos14,Loos15}. 
	
	One important hermitian cubic norm structure is the ``universal" hermitian cubic norm structure on one generator introduced in Proposition \ref{prop: universal model} (it is only universal under the condition that the base ring is be a domain). We implemented this hermitian cubic norm structure on a computer, and used this model to prove Lemma \ref{lem: universal model one invertible} which is later used to prove that an element, in an arbitrary hermitian cubic norm structure, is one-invertible if and only if the quartic norm of the element is invertible.
	Details regarding the implementation, what we specifically computed\footnote{A file with the implementations and performed computations can be found on \url{https://github.com/MichielSmet/hermitian_cubic_norm}.} can be found in Appendix \ref{app A}. 
	
	\subsection{Outline}
	
	In section 2, we build towards some sort of universal hermitian cubic norm structure on a single generator.
	
	Section 3 begins by recalling some basic constructions related to hermitian cubic norm structures. Using these basic constructions, we derive the quartic norm and relate it to the notion of one-invertibility.
	
	In the fourth section, we examine the associated Moufang set using elementary methods. An alternative way to obtain this Moufang set would be as the typical Moufang set associated with the rank one algebraic group.
	
	In the final section, we present our classification of division hermitian cubic norm structures.
	
	Two appendices are included. The first explains how one can implement the model of Proposition \ref{prop: universal model} using basic computer algebra software and precisely which computations were executed using this implementation. The second contains some additional information on one-invertibility for another class of structurable algebras.

	\section{Hermitian cubic norm structures with $1$ generator}
	
	We work over a commutative unital ring $R$. By the category of $R$-algebras, often denoted as $R\textbf{-alg}$, we mean the category of commutative, unital, associative $R$-algebras. To any $R$-module $M$, we associate the functor $K \longmapsto K \otimes M$, which maps from $R\textbf{-alg}$ to $\textbf{Set}$.
	
	\subsection{Definition and direct consequences}

	\begin{definition}
		Consider two $R$-modules $M$ and $N$. A functor $f : M \longrightarrow N$ is called a \textit{polynomial law}. It is said to be \textit{ homogeneous of degree} $i$ if $f_L(\lambda m) = \lambda^i f_L(m)$ for all $\lambda \in L$ and $m \in L \otimes M$. 
		More background on polynomial laws can be found in \cite[section 12]{Skip2024}; the essential facts needed will be introduced directly.
	\end{definition}

	\begin{remark}
		To motivate the term polynomial law, note that 
		\[ f_{L[t]}(t m) = \sum_{i = 0}^n t^i f_{i,L}(m)\]
		for some homogeneous polynomial laws $f_{i,\cdot}$. To a homogeneous polynomial law of degree $n$, we can associate \textit{linearisations} $f^{(i,j)}$, satisfying
		\[ f_L(\lambda m + \mu n) = \sum_{i + j = n} \lambda^i \mu^j f^{(i,j)}_{L}(m,n).\]
		We remark that $f^{(i,j)}$ is homogeneous of degree $i$ in the first component, of degree $j$ in the second component, and that $f^{(n,0)}(a,b) = f(a)$ while $f^{(0,n)}(a,b) = f(b)$.
		Moreover, one has $f^{((i,j),k)} = f^{(i,(j,k))}$ as these can be recovered from \[f_{L[t_1,t_2,t_3]}(t_1 m + t_2 n + t_3 o) = \sum_{i + j + k = n} t_1^i t_2^j t_3^k f^{(i,j,k)}(m,n,o).\]
	\end{remark}

	We use the definition of hermitian cubic norm structures given in \cite[Definition 4.1.1]{Michiel2025}.
	
	\begin{definition}
		Let $R$ be a ring and take $\alpha \in R$ such that $1 - 4 \alpha $ is invertible. Set $K = R[t]/(t^2 - t + \alpha)$ and consider the involution $k \mapsto \bar{k}$ which is $R$-linear and maps $t$ to $1 - t$.
		Let $J$ be a $K$-module endowed with the following maps:
		\begin{itemize}
			\item a $K$-cubic map $N : J \longrightarrow K$,
			\item an $R$-quadratic map $\sharp : J \longrightarrow J$ such that $(k j)^\sharp = \bar{k}^2 j$ for all $k \in K$ and $j \in J$,
			\item the linearisation $(a , b) \mapsto a \times b$ of $\sharp$ which is $\bar{\cdot}$-antilinear in both arguments,
			\item a hermitian map $T : J \times J \longrightarrow K$.
		\end{itemize}
	For our convenience, we will call such a map $\sharp$ $K$-\textit{anti-quadratic}.
	This quadruple $(J,N,\sharp,T)$ is called a \textit{hermitian cubic norm structure} over $K/R$ if 
	\begin{enumerate}
		\item \label{ax1} $T(a,b^\sharp) = N^{(1,2)}(a,b)$,
		\item \label{ax2}  $(a^\sharp)^\sharp = N(a)a$
		\item \label{ax3} $(a^\sharp \times b) \times a = \overline{N(a)} b + T(b,a) a^\sharp$,
		\item \label{ax4} $N( T(a,b)a - a^\sharp \times b) = N(a)^2 \overline{N(b)}$.
	\end{enumerate}
hold over all scalar extensions $L[t]/(t^2 - t + \alpha)$ of $R[t]/(t^2 - t + \alpha)$ induced by extending $R$ to $L$.

In order to avoid confusion, we will consistently use $k a\times b$ for $k(a \times b) = (\bar{k} a) \times b$.
	\end{definition}

	Two important consequences are $v \times v = 2 v^\sharp$ and $T(v,v^\sharp) = 3 N(v)$, which follow immediately from the maps being quadratic or cubic.
A third, important fact is $T(a,b \times c) = T(b, a \times c)$ since both are the $(1,1,1)$-linearisation of $N$.
These three facts will be constantly used without further explanation. 

Before we start, it is worth noting that the last two properties (\ref{ax3}) and (\ref{ax4}) are basically consequences of the first two properties. 

	\begin{lemma}
		\label{lem: redundant axs}
		Consider $(J,N,\sharp, T)$ as in the definition of a hermitian cubic norm structure and assume that (\ref{ax1}) and (\ref{ax2}) are satisfied and that there exists a $j \in J$ with $N(j) \in K^\times$. Property (\ref{ax3}) is then automatically satisfied. Moreover, if there also exists a $k \in J$ for which $f(k) \in K^\times$ for some $K$-linear $f : J \longrightarrow K$, (\ref{ax4}) holds as well.
		\begin{proof}
			This can be shown using \cite[Proposition 12.24]{Skip2024}, which implies that any identity which holds for all $(a,b) \in J_L \times J_L$ with $N_L(a) f_L(b) \in L^\times$ for all scalar extensions $L$, will hold for all elements $(a,b)$, using the condition imposed in the moreover-case. Moreover, if it holds for all $(a,b) \in J_L \times J_L$ with only $N(a) \in L^\times$, it holds also for all $a,b$.
			
			Namely, (\ref{ax3}) follows from the $(3,1)$-linearisation of (\ref{ax2}) applied to $(a^\sharp,b)$ whenever $N(a)$ is invertible (no assumption on $b$ is necessary). \cite[Proposition 12.24]{Skip2024} allows us to say that it holds in general.
			
			To obtain (\ref{ax4}), one can adapt the proof of \cite[33.8]{Skip2024} which proves the same for cubic norm structures. Namely, one can show that $Q_{a^\sharp} Q_a b = \overline{N(a)}^2 b$, for all $a, b$. For the $b$ for which $f_L(b) \in (K \otimes L)^\times$, one can combine this with $(Q_a b)^\sharp = Q_{a^\sharp} b^\sharp$ to obtain \begin{align*}
				\overline{N(a)}^4 N(b) f_L(b)  & = f_L(Q_{a^\sharp} N(a)^2 \overline{N(b)} Q_a b) \\ & = f_L(Q_{a^\sharp} (T(a^{\sharp\sharp},b^{\sharp\sharp}) a^{\sharp\sharp} - a^{\sharp\sharp\sharp} \times b^{\sharp\sharp})\\ & =  f_L(Q_{a^\sharp} (Q_a b)^{\sharp \sharp}) \\ &= f_L(Q_{a^\sharp} N(Q_a b) Q_a b)  \\ & = \overline{N(Q_a b)} \overline{N(a)}^2 f_L(b),
			\end{align*}  which implies that (\ref{ax4}) holds under the stated conditions, using \cite[Proposition 12.24]{Skip2024} in the way indicated above.			
		\end{proof}
	\end{lemma}

	\subsection{One generator hermitian cubic norm structures}

	We assume that we are working with a fixed hermitian cubic norm structure. We first identify what happens if we look at the hermitian cubic norm substructure generated by a single element, enabling the construction of a universal model. This model will play a crucial role in proving Lemma \ref{lem: universal model one invertible}, which underpins the subsequent results.
	
	We shall see that the hermitian cubic norm structure generated by $v$ over $K/R$ depends on $2$ elements in $R$ and one $K$, namely $T(v,v), T(v^\sharp,v^\sharp)$, which are hermitian and thus lie in $R$, and $N(v)$ which is just an element of $K$.
	
	\begin{lemma}
		\label{lem: def sharp}
		Suppose that $v \in J$. The $K$-submodule spanned by $v , v^\sharp$ and $v \times v^\sharp$ is closed under $\sharp$.
		Moreover, restricted to $\langle v, v^\sharp, v \times v^\sharp \rangle$ the map $\sharp$ is the unique $K$-anti-quadratic map $f$ such that
		\begin{enumerate}
			\item $f(v) = v^\sharp$,
			\item $f(v^\sharp) = N(v) v$,
			\item $f(v \times v^\sharp)= - \overline{N(v)} v \times v^\sharp + \overline{N(v)} T(v,v) v + T(v^\sharp,v^\sharp) v^\sharp$,
			\item $f^{(1,1)}(v, v^\sharp) = v \times v^\sharp$
			\item $f^{(1,1)}(v , v \times v^\sharp)  = \overline{N(v)} v + T(v,v) v^\sharp$,
			\item $f^{(1,1)}(v^\sharp, v \times v^\sharp) = N(v) v^\sharp + T(v^\sharp,v^\sharp) v$,
		\end{enumerate}
	i.e., $f(av + bv^\sharp + c v \times v^\sharp)$ equals
	\begin{equation}\bar{a}^2 f(v) + \overline{ab} f^{(1,1)}(v,v^\sharp) + \bar{b}^2 f(v^\sharp) + \overline{ac} f^{(1,1)}(v,v \times v^\sharp) + \overline{bc} f^{(1,1)}(v^\sharp,v \times v^\sharp) + \bar{c}^2 f(v \times v^\sharp).\end{equation}
	\end{lemma}
\begin{proof}
	Equations (1), (2), and (4) are obvious.
	Taking the $(3,1)$-linearisation of axiom (\ref{ax2}), yields
	\[ (a \times b) \times a^\sharp = N(a) b + T(b,a^\sharp) a.\]
	Setting $b = v^\sharp$, $a = v$, yields
	\[ (v \times v^\sharp) \times v^\sharp = N(v)v^\sharp + T(v^\sharp,v^\sharp) v \]
	which is precisely (6). Setting $a = v = b$ in axiom (\ref{ax3}) yields
	\[ (v \times v^\sharp) \times v = \overline{N(v)} v + T(v,v) v^\sharp,\]
	which is precisely (5).
	Lastly, taking the $(2,2)$-linearisation of axiom (\ref{ax2}) yields
	\[ a^\sharp \times b^\sharp + (a \times b)^\sharp = T(a,b^\sharp)a + T(b,a^\sharp)b.\]
	Setting $a = v, b = v^\sharp$ yields
	\[ \overline{N(v)} v \times v^\sharp + (v \times v^\sharp)^\sharp = \overline{N(v)} T(v,v) + T(v^\sharp,v^\sharp) v^\sharp,\]
	which implies (3).
\end{proof}

\begin{lemma}
	\label{lem: N on basis}
	For all $v$, we have $N(v \times v^\sharp) = N(v)(T(v,v) T(v^\sharp,v^\sharp) - N(v)\overline{N(v)})$.
	If $K$ is a domain, we have $N(v^\sharp) = \overline{N(v)}^2$ for all $v$.
	\begin{proof}
		Axiom (\ref{ax4}) applied to $a = v$ and $b = v^\sharp$ yields
		\[ N( T(v,v^\sharp) v - v^\sharp \times v^\sharp) = N(v)^2 \overline{N(v^\sharp)},\]
		simplifying the left-hand side yields
		\[ N( 3N(v) v - 2 N(v) v) = N(N(v)v) = N(v)^4.\]
		If $N(v) \neq 0$ and $K$ is a domain, we conclude that $N(v^\sharp) = \overline{N(v)}^2$.
		If $N(v) = 0$, we can play the same game with $v^\sharp$ to see that $N(v^\sharp)^2 = \overline{N(v)}^4 = 0$ so that $N(v^\sharp) = 0$.
		
		The other equation also follows from axiom (\ref{ax4}), applied with $a = b = v$, since it implies
		\[ N(T(v,v) v - v^\sharp \times v) = N(v)^2 \overline{N(v)}.\]
		We expand the left-hand side as
		\begin{align*}
			N(T(v,v) v - v^\sharp \times v) = & \;T(v,v)^3 N(v) - T(v,v)^2 T(v^\sharp \times v, v^\sharp) + T(v,v) T(v, (v \times v^\sharp)^\sharp)\\& - N(v \times v^\sharp) \\
			= & \;T(v,v)^3 N(v) - 2 T(v,v)^3 N(v) - N(v \times v^\sharp) \\
			& + T(v,v) T(v, - \overline{N(v)} v^\sharp \times v + \overline{N(v)} T(v,v) v + T(v^\sharp, v^\sharp) v^\sharp) \\
			= &  - T(v,v)^3 N(v) - N(v \times v^\sharp)
			 - 2 N(v) T(v,v)T(v^\sharp,v^\sharp)\\ &+ T(v,v)^3 N(v)   + 3 N(v) T(v^\sharp,v^\sharp) T(v,v) \\
			= & \; N(v) T(v,v) T(v^\sharp,v^\sharp) - N(v \times v^\sharp).\\
		\end{align*}
		Hence, we conclude 
		\[ N(v \times v^\sharp) = N(v)( T(v,v) T(v^\sharp,v^\sharp)- N(v)\overline{N(v)}). \qedhere \]
	\end{proof}
\end{lemma}

\begin{lemma}
The map $T$ on $\langle v ,v^\sharp, v \times v^\sharp \rangle$ is the unique hermitian map $h$ such that 
	\begin{enumerate}
		\item $h(v,v) = T(v,v)$,
		\item $h(v,v^\sharp) = 3 N(v)$,
		\item $h(v, v \times v^\sharp) = 2 T(v^\sharp,v^\sharp)$,
		\item $h(v^\sharp, v \times v^\sharp) = 2 \overline{N(v)} T(v,v)$,
		\item $h(v \times v^\sharp, v \times v^\sharp) = T(v,v) T(v^\sharp,v^\sharp) + 3 N(v)\overline{N(v)}$,
		\item $h(v^\sharp,v^\sharp) = T(v^\sharp,v^\sharp)$
	\end{enumerate}
Moreover, given fixed constants $T(v,v), T(v^\sharp,v^\sharp) \in R$ and $N(v) \in K$, there exists such a hermitian map $h$.
	\begin{proof}
		The only non-obvious equation is (5).
		This equation follows from \begin{align*}
			T(v \times v^\sharp, v \times v^\sharp) & = T(v, (v \times v^\sharp) \times v^\sharp) \\ & = T(v, N(v) v^\sharp + T(v^\sharp,v^\sharp) v) \\& = 3 N(v)\overline{N(v)} + T(v^\sharp,v^\sharp) T(v,v). \qedhere
		\end{align*}
	\end{proof}
\end{lemma}

Given the previous three lemmas, we can also uniquely reconstruct the cubic norm $N$ on the module $\langle v , v^\sharp , v \times v^\sharp \rangle$ starting from arbitrary values $T(v,v) ,T(v^\sharp,v^\sharp) \in R$ and $N(v) \in K$, assuming that $N(v^\sharp) = \overline{N(v)}^2$.
To be precise, it is given by

\begin{align}
	\label{def: N}
	N(av + bv^\sharp + c v \times v^\sharp) = & \; a^3 N(v) + b^3 \overline{N(v)}^2 + a^2b T(v^\sharp,v^\sharp) + ab^2 \overline{N(v)} T(v,v) \\ & + c^3 N(v)(T(v,v)T(v^\sharp,v^\sharp) - N(v)\overline{N(v)}) \nonumber \\ & + c^2 T(av + b v^\sharp, (v \times v^\sharp)^\sharp) + c T(v \times v^\sharp, (av + b v^\sharp)^\sharp) \nonumber.
\end{align}

So, now we are ready to define the ``universal" hermitian cubic norm structure.
Set $R = \mathbb{Z}[x,y,\alpha, N_1, N_2, (1 - 4 \alpha)^{-1}]$ and $K = R[t]/(t^2 - t + \alpha)$.
Let $U$ be the hermitian cubic norm structure over $K/R$, formed by the free module of rank $3$ generated by $v, v^\sharp, v \times v^\sharp$.
Use the previous lemmas and formula for $N$ to define the unique $\sharp, T, N$ such that
\[ T(v,v) = x, \quad T(v^\sharp,v^\sharp) = y, \quad N(v) = N_1t + N_2(1 - t).\]

\begin{remark}
	In Appendix \ref{app A}, we explain how we implemented the hermitian cubic norm structure on a computer.
	The precise formulas for $T$,  $\sharp$ and $N$ are restated there.
	Results involving this model are explained in some detail over there.
\end{remark}

\begin{definition}
	Consider two commutative unital rings $R$ and $R'$, $K = R[t]/(t^2 - t + \alpha)$, a ring homomorphism $f : R \longrightarrow R'$, $K' \cong R'[t]/(t^2 - t + f(\alpha))$, a hermitian cubic norm structure $J$ over $K/R$, and a hermitian cubic norm structure $J'$ over $K'/R'$.
	We remark that $J'$ is a $K$-module, using the homomorphism $K \longrightarrow K'$ induced by $f$.
	Each $K$-module homomorphism $g : J \longrightarrow J'$ that preserves $N, \sharp$, and $T$, is called a \textit{hermitian cubic norm structure homomorphism}. 	
\end{definition}

\begin{proposition}
	\label{prop: universal model}
	Consider $R = \mathbb{Z}[x,y,\alpha, N_1, N_2, (1 - 4 \alpha)^{-1}]$ and $K = R[t]/(t^2 - t + \alpha)$. Define $U$ as the free $K$-module generated by $v, v^\sharp$ and $v \times v^\sharp$. There exists a unique hermitian cubic norm structure on $U$ such that $T(v,v) = x, T(v^\sharp,v^\sharp) = y$ and $N(v) = N_1 t + N_2(1 - t)$. Moreover, for any hermitian cubic norm structure $J$ over a domain $K' = R'[t]/(t^2 - t + \beta)$ and any $w \in J$, there exists a hermitian cubic norm structure homomorphism $U \longrightarrow J$; the assumption that $K'$ is a domain can be dropped if one requires that $N(w^\sharp) = \overline{N(w)}^2$.
	
	\begin{proof}
		This is Lemma \ref{lem: app2}, if one observes that the previous $3$ lemmas prove the necessity of the formulas to define $N, \sharp$ and $T$. The moreover-part follows similarly from the necessity of the formulas employed to define $N, \sharp$ and $T$.
	\end{proof}
\end{proposition}

\section{The associated structurable and Lie algebras}

In this section, we recall the construction of the structurable and Lie algebras associated to a hermitian cubic norm structure.
We then introduce the notion of being a division structurable algebra and apply the one generator hermitian cubic norm structure to study them.

\subsection{Basic constructions}

We recall some constructions developed in \cite{Michiel2025} related to hermitian cubic norm structures.
\begin{construction}
	With a hermitian cubic norm structure $J$ over $K/R$, we can associate an algebra $A = K \oplus J$ with multiplication
	\[ (k_1,j_1)(k_2,j_2) = (k_1k_2 + T(j_1,j_2), k_1j_2 + \bar{k}_2 j_1 + j_1 \times j_2)\]
	and involution
	\[ \overline{(a,b)} = (\bar{a},b).\]
	We call this algebra the \textit{associated structurable algebra}.
	
	On this algebra, we can define \[V : A \times A \times A  \longrightarrow A: (a,b,c) \mapsto - (a\bar{b}) c - (c \bar{b})a + (c \bar{a}) b\] and we shall use $V_{a,b} c$ to denote $V(a,b,c)$.
	The $R$-linear span \[\mathfrak{instr}(J) = \langle (V_{a,b}, - V_{b,a}) | a, b \in A \rangle_R \subset \End(A) \times \End(A)\] is the \textit{inner structure algebra}, if one uses the operation
	\[ [f,g] = f \circ g - g \circ f.\] 
	The inner structure algebra is closed under this composition since
	\[ [V_{a,b}, V_{c,d}] = V_{V_{a,b} c, d} - V_{c, V_{b,a} d}.\]
\end{construction}

\begin{remark}
	One usually defines $V_{a,b} c$ as the minus our $V_{a,b} c$. In \cite{Michiel2025}, we also use the opposite sign convention (but a slightly different version of the Lie algebra construction that immediately resolves this discrepancy). The description of the Lie algebra we use here, is more in line with \cite{All99}.
\end{remark}

\begin{construction}
	Using the associated structurable algebra $A$ and inner structure algebra $I$ for a hermitian cubic norm structure over $R[t]/(t^2 - t + \alpha)$, we can define a $\mathbb{Z}$-graded Lie algebra $L = \bigoplus_{i \in \mathbb{Z}} L_i$ with $L_i = 0$ if $|i| > 2$ and
	\begin{itemize}
		\item $L_2 \cong L_{-2} = R(t - \bar{t})$,
		\item $L_1 \cong L_{-1} = A$,
		\item $L_0 = I$.
	\end{itemize}
To distinguish the distinct copies of $A$ and $R(t - \bar{t})$ we use $A_1, A_{-1}$ and $R(t - \bar{t})_2, R(t - \bar{t})_{-2}$.
The operation is the unique one such that
\begin{itemize}
	\item $I$ is a subalgebra
	\item $(f, g) \in I$ acts on $(x,y) \in A_{1} \times A_{-1}$ as $[(f,g),(x,y)] = (fx, gy)$.
	\item $[(x_1,y_1),(x_2,y_2)] = (x_1\bar{x_2} - x_2 \bar{x_1})_2 + (V_{x_1,y_2} - V_{x_2,y_1}, - V_{y_2,x_1} + V_{y_1,x_2}) + (y_1 \bar{y_2} - y_2 \bar{y_1})_{-2}$.
\end{itemize}
This implies that $[y_{-1},s_2] = (sy)_1$, $[(V_{x,y}, - V_{y,x}), s_2] = [x_1, (sy)_1]$, and $[s_{2}, [s'_{-2}, x_1]] = (s(s'x))_1$ for $s, s' \in R(t - \bar{t})_{\pm 2}$ and $(x_1,y_{-1}) \in A_1 \times A_{-1}$.
\end{construction}

\begin{construction}
We also have specific automorphisms of the Lie algebra.
Set \[G = \{ ((a,v), (u, av + v^\sharp)) \in (K \times J)^2 | u + \bar{u} = a\bar{a} + T(v,v)\}.\]
There are an actions $\exp_\pm$ of $G$ on $L$ such that
\[ \exp_+(x,y) = 1 + (x,y)_1 + (x,y)_2 + (x,y)_3 + (x,y)_4\]
for each $(x,y) \in A^2 \cap G$ with $(x,y)_i$ acting as $+i$ on the grading and $(x,y)_1 = \ad x$.
Moreover, these automorphisms are uniquely determined by
\[ (x,y)_2 c_{-1} = ((x \bar{c}) x - y c)_{1} = (Q_{(x,y)} c)_{1}\]
and
\[ (x,y)_3 c_{-1} = (T_{(x,y)} c)_{2}\]
for a specific function $T_{(x,y)}$, since $L$ is generated by $A_{1}$ and $A_{-1}$. For a formula of $T$, see \cite[Theorem 4.2.1]{Michiel2025}.
The action $\exp_-$ is ``the same action" with the grading reversed.

Moreover, if we look at $\langle \exp_+(G), \exp_-(G) \rangle \subset \Aut(L)$ we can recover $\exp_+(G)$ as the elements $e$ of the form
\[ 1 + \sum_{i = 1}^4 e_i\]
with $e_i$ acting as $+i$ on the grading. We call the group $\langle \exp_+(G), \exp_-(G) \rangle$ \textit{the projective elementary group}. We also use $G^+$ to denote $\exp_+(G)$ and $G^-$ for $\exp_-(G)$. To denote the subgroup of $\langle \exp_+(G), \exp_-(G) \rangle$ of elements that preserve the grading, we use $\Inn(G^+,G^-)$.
\end{construction}

\begin{lemma}
	\label{lem: def nu}
	For $(x,y) = ((a,v),(u, av + v^\sharp)) \in G$, we have \[\exp_+(x,y) (s)_{-2} = \nu(u,a,v) s_{2} \mod L_{< 2}.\]
	This scalar $\nu(u,a,v) \in R$ is given by
	\[ u \bar{u} - a\bar{a} T(v,v) + a N(v) + \bar{a} \overline{N(v)} - T(v^\sharp,v^\sharp).\]
	\begin{proof}
		This is a lengthy computation, feasible by hand.
		First, we can extend scalars with a faithfully flat $R'$, in such a way that $R'[u]/(u^2 - u + \alpha) \cong R'[t]/(t^2 - t)$, which can be done by \cite[Theorem 4.1.7]{Michiel2025}.
		
		Using that $\exp_+(x,y)$ is an automorphism, we obtain
		\[\exp_+(x,y) [t, \bar{t}] = [(T_{(x,y)} t),[x,\bar{t}]] + [[x,t],(T_{x,y} \bar{t})]  + [Q_{(x,y)} t, Q_{(x,y)} \bar{t}] \mod L_{<2}.\]
		
		One computes that
		\begin{align*}
			[(T_{(x,y)} t),[x,\bar{t}]] + [[x,t],(T_{x,y} \bar{t})] = (t - \bar{t})(a\bar{a}(a\bar{a} - T(v,v)) + a N(v) + \bar{a} \overline{N(v)}).\end{align*}
		One can also show that
		\begin{align*}
			[Q_{(x,y)} t, Q_{(x,y)} \bar{t}] &= (t - \bar{t})( - (a\bar{a})^2 + u \bar{u} - T(v^\sharp,v^\sharp)).
		\end{align*}
	Combining those two yields
	\[ \nu(u,a,v) = u \bar{u} - a\bar{a} T(v,v) + aN(v) + \bar{a} \bar{N(v)} - T(v^\sharp,v^\sharp). \qedhere\]
	\end{proof}
\end{lemma}

\begin{definition}
	The natural transformation $\nu : G^+ \longrightarrow R$ of the previous lemma will be called the \textit{quartic norm}. We will use the following notation and definition:
	\[ \nu((a,v),(u,av + v^\sharp)) = \nu(u,a,v) = u \bar{u} - a\bar{a} T(v,v) + a N(v) + \bar{a} \overline{N(v)} - T(v^\sharp,v^\sharp).\]
	We call a hermitian cubic norm structure for which the quartic norm is invertible for all $g \in G^+(R)$ \textit{division}. This immediately implies that $R$ is a field, since $\nu(r(t - \bar{t}),0,0) = r^2(1 - 4 \alpha)$ is invertible for all $r \in R \setminus \{ 0\}$.
	
	We will call a hermitian cubic norm structure anisotropic to indicate that the cubic norm is anisotropic. We use division instead when the quartic norm is anisotropic, as it is related to the notion of a ``division structurable algebra".
\end{definition}

\begin{remark}
	A priori, the cubic and quartic norm do not tell that much about each other. On one hand, one can have division hermitian cubic norm structures for which the cubic norm is identically zero (these will appear when we classify division hermitian cubic norm structures). On the other hand, one can have anisotropic hermitian cubic norm structures which are not division, such as $K$ over $K/R$ with $N(x) = x^3$ and $x^\sharp = \bar{x}^2$, for which
	\[ \nu((\lambda, - \lambda),(2 \lambda^2, 2 \lambda^2)) = 4 \lambda^4 - \lambda^4 - \lambda^4 - \lambda^4 - \lambda^4 = 0\]
	for all $\lambda \in R$.
\end{remark}

\begin{remark}
	If one fills in $u = a\bar{a}/2 + T(v,v)/2 + \lambda^2 (t - \bar{t})$ in $\nu(u,a,v)$, one obtains
	\[ \nu(u,a,v) = 1/4((a\bar{a} - T(v,v))^2 + 4aN(v) + 4\bar{a}\overline{N(v)} - 4T(v^\sharp,v^\sharp)) + \lambda^2(1 - 4 \alpha),\]
	the quartic norm one typically associates to a skew dimension one structurable algebra with an additional term $\lambda^2(1 - 4 \alpha)$.
	It is known over fields of characteristic different from $2$ and $3$ that $\nu( a\bar{a}/2 + T(v,v)/2,a,v) \neq 0$ if and only if the element $(a,v) \in A$ is \textit{conjugate invertible}, see, e.g., \cite{Tom2019, Allison84}. To allow the additional degree of freedom $\lambda$, one has to combine this knowledge with \cite{Tom}. Below, we shall prove that $\nu(u,a,v)$ being invertible is necessary and sufficient to be \textit{one-invertible}, which is the generalization employed in \cite{Tom}.
	
	We remark that $\Inn(G^+,G^-) \subset \langle G^+ ,G^- \rangle $ formed by the elements that preserve the grading, will be a subgroup of the similitudes of the norm $\nu$. In the case that our hermitian cubic norm structure is related to an Albert algebra, one can relate this to Freudenthals construction of $E_7$ \cite{Freu54}, as explained in an easily recognizable manner in \cite[Proposition 9.18]{Jacobson71}. This quartic norm plays an important role when considering groups and Lie algebras of type $E_7$, examples except the ones already provided include \cite{Brown69, Skip2001, Alsaody21, Skip2023, Alsaody24}. It might thus be interesting to ask to what degree $\Inn(G^+,G^-)$ is determined by the group off similitudes of $\nu$ on $G$, first in the $E_7$ (hermitian cubic norm structure coming from an Albert algebra) case and thereafter in general. In \cite[Theorem 8.8]{Muhlherr19}, this group of similitudes of $\nu(u,a,v)$ was determined (over fields of arbitrary characteristic) for split hermitian cubic norm pairs coming from cubic norm structures; in this case $\Inn(G^+,G^-)$ coincides with the group of similitudes if and only if the inner structure group and the structure group for the cubic norm structure coincide.
	
	We remark that this quartic norm and the relation to $E_7$ inspired the axiomatization of Freudenthal triple systems \cite{Meyberg68} and \cite[section 4]{Brown69}.
	To classify the possible Freudenthal triple systems, Meyberg proved that, assuming the existence of a certain idempotent, one can construct the triple system from a Jordan algebra of degree $\le 3$.
	Later, when structurable algebras were defined, one could see skew dimension one structurable algebras and Freudenthal triple systems as being basically equivalent, see e.g. \cite[Theorem 3.17 and Theorem 5.2]{Boel13}
\end{remark}

\subsection{A necessary and sufficient criterion for being a division hermitian cubic norm structure}

The goal of this subsection is to prove that if $R$ is arbitrary and if the quartic norm $\nu : G\setminus\{0\} \longrightarrow R^\times$ takes nonzero values, that each $g \in G$ is \textit{one-invertible}, i.e., there exist $g_r$ and $g_l \in G$ such that
\begin{equation} \exp_-(g_l) \exp_+(g) \exp_-(g_r) \label{eq: oneinvert} \end{equation}
reverses the grading of the Lie algebra. We will also use this to establish that $K/R$ is a quadratic field extension in that case.

\begin{lemma}
	\label{lem: universal model one invertible}
	Let $U$ be the hermitian cubic norm structure of Proposition \ref{prop: universal model} over $\mathbb{Z}[x,y,\alpha, N_1, N_2, (1 - 4 \alpha)^{-1}][t]/(t^2 - t + \alpha)$ and extend the scalars with $\mathbb{Z}[z_1, z_2, z_3]$. 
	
	We set $a = z_1 t + z_2 (1 - t)$ and $u = (z_1z_2 + T(v,v) + z_3)t - z_3 \bar{t}$ and note that $g = ((a,v), (u, av + v^\sharp)) \in G^+$.
	
	If we add $\nu(u,a,v)^{-1}$ to the ring of scalars, $g$ is one-invertible.
	\begin{proof}
		We prove this lemma in the appendix, cfr., Lemma \ref{lem: app main}.
		In this lemma, we determine that $g_r = (\alpha, \beta)$ is determined by
		\begin{equation}\nu(u,a,v) \alpha  = (- u a + a^2 \bar{a} + \overline{{N(v)}}, (- u + T(v,v))v + \bar{a} v^\sharp - v \times v^\sharp).\label{def gr} \end{equation}
		and
		\begin{equation}
			\nu(u,a,v)^2 \beta = (\gamma,w)
		\end{equation}
			with $w = \nu(u,a,v)^2(\alpha_1 \alpha_2 + \alpha_2^\sharp)$ and 
		\[ \gamma = \nu(u,a,v) u + 2 N(v) \overline{N(v)} + 2 a\bar{a} T(v^\sharp,v^\sharp) - 2 (u aN(v) + \bar{u} \bar{a} \overline{N(v)}). \qedhere \]
	\end{proof}
\end{lemma}

	\begin{proposition}
		\label{prop: g one inv iff nu inv with ass}
		Suppose that $J$ is an arbitrary hermitian cubic norm structure and let $g = ((a,v),(u, av + v^\sharp)) \in G$ and suppose that $\overline{N(v)^2} = N(v^\sharp)$.
		Then $g$ is one-invertible if and only if $\nu(g)$ is invertible.
		\begin{proof}
			If $g$ is one-invertible, Lemma \ref{lem: def nu} implies that $\nu(g)$ is invertible.
			
			Now, we want to prove the converse.
			We want to apply Lemma \ref{lem: universal model one invertible}. We note that the assumption $\overline{N(v)}^2 = N(v^\sharp)$ guarantees the existence of a hermitian cubic norm structure homomorphism from the hermitian cubic norm structure of Proposition \ref{prop: universal model} which maps the one-invertible element of Lemma \ref{lem: universal model one invertible} of previous lemma to $g$. We also use $g_l$ and $g_r$ to denote the homomorphic images of the $g_l$ and $g_r$ obtained from Lemma \ref{lem: universal model one invertible}.
			
			We remark that $\tau = g_l g g_r$ acts on the Lie algebra of the hermitian cubic norm structure of Lemma \ref{lem: universal model one invertible} and on the Lie algebra of $J$ and that these actions are preserved. So, we use that $\tau$ acts on  $M = \langle (t - \bar{t})_{2}, (\text{Id}, - \text{Id}), (t - \bar{t})_{-2} \rangle$ as the matrix
			\[ \begin{pmatrix}
				0 & 0 & \nu \\ 0 & - 1 & 0 \\ \nu^{-1} & 0 & 0
			\end{pmatrix}\]
		using the ordered basis using to describe $M$. The second and third column are obtained using the definition of $\nu$ and the fact that $\tau$ reverses the grading of this algebra. To obtain the first column, use the reversion of the grading combined with the fact that the determinant should equal $1$.
		This also implies that
		\[ \tau \exp_-(0,t - \bar{t}) \tau^{-1} = \exp_+(0,\nu(g) (t - \bar{t})), \]
		since 
		\[ \exp_-(0,t - \bar{t}) = 1 + e_2 + e_4\]
		with $e_2 = \ad (t - \bar{t})_{-2}$ and $e_4$ the unique endomorphism of $L$ that acts as $-4$ on the grading and satisfies
		\[ e_4 [x,y] = [e_2 x, e_2 y]\]
		for all $x,y$ in the Lie algebra.
		
		Now, $r \mapsto r_+= \exp_+(0,r(t - \bar{t}))$ and $r \mapsto r_- = \exp_-(0, (t - \bar{t})/(1 - 4\alpha))$ form a divided power representation of the Jordan pair $(R,R)$ with $Q_x y = x^2 y$ (definition in \cite{faulkner2000jordan}).
		In particular,
		\[ \beta(\mu,\lambda) = (- \lambda/(1 - \mu\lambda))_- \mu_+ \lambda_- (- \mu/ (1 - \mu\lambda))_+\]
		is $0$ graded by the exponential property proved in \cite{faulkner2000jordan} over $R[[\mu,\lambda]]$.
		Using this fact, one can show that $\beta(\mu,\lambda) v = (1 - \mu\lambda)^{i} v$ for $v \in L_i$.
		For $i = 1$ this is a straightforward calculation; for $i = -1$ the same straightforward calculation works for $\beta(\mu,\lambda)^{-1}$; for other $i$ it follows from the fact that $\beta(\mu,\lambda)$ is an automorphism and that $L$ is generated by $A_1$ and $A_{-1}$.
		
		We note that $\beta(\mu,\lambda)$ exists and acts the same for each invertible $1 - \mu \lambda$.
		Moreover, we have $\beta(\alpha \mu, \lambda) = \beta(\mu, \alpha \lambda)$ for all invertible $1 - \alpha \mu \lambda$.
		Now, 
		\[ \tau \beta(\mu, \lambda) \tau^{-1} = \beta^{-1}(- \nu(g) \lambda, - \nu(g)^{-1} \mu) = \beta^{-1}(\lambda, \mu).\]
		This shows that
		\[ \beta(1,\lambda\mu) \tau \beta(1,\lambda\mu) = \tau.\]
		We conclude that $\tau$ reverses the grading.
		\end{proof}
	\end{proposition}

\begin{lemma}
	\label{lem: gr oneinv}
	Suppose that $g$ is one-invertible, then $g_l$ and $g_r$ are one-invertible as well. Moreover, setting $\tau_{g,\pm} = \exp_\mp(g_l) \exp_\pm(g) \exp_\mp(g_r)$ we have $\tau_{g,+} = \tau_{g_l,-} = \tau_{g_r,-}$ and
	\begin{itemize}
		\item $(g_l)_r = g$,
		\item $(g_l)_l = \tau_g^{-1}(g_r)$,
		\item $(g_r)_l = g$,
		\item $(g_r)_r = \tau_g (g_l)$.
	\end{itemize}
	\begin{proof}
		Everything follows from
		\[ \tau_{g,+} = g g_r g_r^{-1} g^{-1} g_l^{-1} g_l g_l g g_r = g g_r \tau_{g,+}^{-1} g_l \tau_{g,+} \]
		and variations thereof, where we dropped the exponentials (all $g_l$, $g_r$ should come with an $\exp_-$ and all $g$ with an $\exp_+$).
		To switch between $G^+$ and $G^-$, one can make use of $\tau_{(\sqrt{2},1)}$ which sends $\exp_+(a,b)$ to $\exp_-(a,b)$, as can be observed from \cite[Remark 5.10]{OKP}.
	\end{proof}
\end{lemma}

\begin{remark}
	We shall often write $\tau_g$ for the $\tau_{g,+}$ defined in the previous lemma.
\end{remark}

Recall that a hermitian cubic norm structure is division if the quartic norm is invertible for all elements in $G \setminus \{1\}$ and that these only exist over fields $R$.

\begin{lemma}
	\label{lem: div implies K field}
	A division cubic norm structure over $K/R$ can only exist if $K = R[t]/(t^2 - t + \alpha)$ is a quadratic field extension.
	\begin{proof}
		Evaluating $\nu$ at $((k,0),(k\bar{k}t, 0))$ yields $(k \bar{k})^2 \alpha$ so that $k \bar{k}$ is invertible for all nonzero $k$, which implies that each $k \in K$ is invertible.
	\end{proof}
\end{lemma}

\begin{proposition}
	A hermitian cubic norm structure is division if and only if each $g \neq 1$ is one-invertible.  
	\begin{proof}
		This is more or less Proposition \ref{prop: g one inv iff nu inv with ass}. Namely, if we have a division hermitian cubic norm structure, Lemmas \ref{lem: div implies K field} and \ref{lem: N on basis} show that $N(v^\sharp) = \overline{N(v)}^2$ so that  we can apply that proposition to conclude that each $g \neq 1$ is one-invertible.
		
		We remark that if $g$ is one-invertible that $s \mapsto \nu(g) s$ is the action of Equation \ref{eq: oneinvert} restricted to $L_{-2} \longrightarrow L_2$, so that $\nu(g)$ is invertible.
	\end{proof}
\end{proposition}

\section{Moufang sets from division hermitian cubic norm structures}

We describe the Moufang set related to a division hermitian cubic norm structures. None of the results obtained in this section are unexpected. The final description we obtain at the end is a straightforward extension of \cite[Theorem 5.1.5]{Tom2019} and the methods employed here can be straightforwardly generalized to \textit{division} structurable algebras in the sense of \cite{OKP}, if one calls such an algebra division whenever each element different from $1$ in the associated operator Kantor pair is one-invertible.

\begin{definition}
	We call a set $M$ equipped with groups $U_m \subset \text{Sym}(M)$ for each $m \in M$ such that
	\begin{enumerate}
		\item $U_m$ acts regularly, (i.e., sharply transitively) on $M \setminus \{m\}$,
		\item for each $g \in U_m$, we have $g U_n g^{-1} \subset U_{g \cdot n}$ for all $n$,
	\end{enumerate}
	a \textit{Moufang set}. The groups $U_m$ are called the root groups and $G(M) = \langle U_m | m \in M \rangle$ is called the \textit{little projective group}.
\end{definition}

\begin{example}
	Consider a \textit{saturated split BN-pair of rank one}, cfr. \cite{Tom09}, i.e., a group $G$ containing subgroups $B$ and $N$ such that
	\begin{enumerate}
		\item $G = \langle B , N \rangle,$
		\item $H = B \cap N$ is a normal subgroup of $N$,
		\item $N = \langle H, \omega \rangle$ for some $\omega \in N$ with $\omega^2 \in H$ and $G = B \cup B \omega B$ and $\omega B \omega \neq B$,
		\item $B \cong H \ltimes U$ for a normal subgroup $U$ ,
		\item $H = B \cap B^\omega$.
	\end{enumerate}
	In this case, $X = \{ U^g | g \in G\}$ is a Moufang set \cite[Proposition 2.1.3]{Tom09} with $U^g$ having itself as root group.
	This immediately implies that any rank one reductive linear algebraic group defines a Moufang set by \cite[(15) and (16)]{Tits64} with $N$ the normalizer of a maximal torus and $B$ a minimal parabolic containing the torus. In the context of algebraic groups, one typically works with the isomorphic Moufang set defined on $X = \langle B^g | g \in G \rangle$.
\end{example}

We will describe the Moufang set structure we obtain from a division hermitian cubic norm structure directly, without referring to algebraic groups. However, since we will later prove that each (finite dimensional) division hermitian cubic norm structure comes from a rank one algebraic group, this would be also a legitimate way to obtain some of the results we prove here.

Given that each parabolic subgroup of an algebraic group is its own normalizer, we prefer to work with $X = G/B$.

Consider a division hermitian cubic norm structure $J$ and consider the automorphism groups $G^\pm = \exp_\pm(G)$ of its Lie algebra. Set $G(J) = \langle G^+, G^- \rangle$ and define the group $P = \Inn(G^+,G^-) \ltimes G^+ \subset G(J)$.
This semidirect product is justified by \cite[Theorem 4.22]{Michiel2025}

\begin{lemma}
	Suppose that $\tau_1$ and $\tau_2$ are two grade reversing automorphisms of $L$, then $\tau_1 \tau_2 \in \Inn(G^+,G^-)$. In particular, for each $g \in G \setminus \{1\}$ and each grade reversing $\tau$ we have
	\[ \tau = (g_l)_+ g_- (g_r)_+ h\]
	for some $h \in \Inn(G^+,G^-)$.	
	Furthermore, $\tau P^+ = g_+ (g_r)_- P^+$ holds for all $g \neq 1$.	
	\begin{proof}
		Trivial.
	\end{proof}
\end{lemma}

\begin{lemma}		
	There exists a bijection $G(J)/P^+ \longrightarrow G^- \cup \{\tau\}$ with $\tau$ a grade reversing automorphism\footnote{Any $g_l g g_r$ suffices as $\tau$.}
	\begin{proof}
		
		We will show that $G^- P^+ \cup \tau P^+ = G$.
		Establishing this is sufficient, since (1) $g_- P^+ \neq h_- P^+$ if and only if $(h^{-1} g)_- \notin P^+$ which is trivially so if $G^- \ni g_- \neq h_- \in G^-$, and since (2) $\tau \notin P^+$ which also shows (3) $h_- \tau \in \tau P^+$ for each $h_- \in G^-$.
		
		Since $G$ is generated by $G^+$ and $\tau$ it is sufficient to show that $G^- P^+ \cup \tau P^+$ is closed under right and left multiplication by $G^+$ and $\tau$. For $G^+$, we have that
		\begin{itemize}
			\item $G^- P^+ G^+ = G^- P^+$,
			\item $\tau P^+ G^+ = \tau P^+$,
			\item $G^+ \tau P^+ \subset \tau P^+ \cap G^- P^+$ since $\tau^{-1} = g^{-1}_+ (g_l^{-1})_- z_+$ with $z_+ = \tau(g_r^{-1})$ for $g_+ \neq 1$ so that 
			\[ g_+ \tau P^+ = g_+ (g^{-1})_+ (g_l^{-1})_- P^+ = (g_l^{-1})_- P^+ \subset G^- P^+\]
			holds in that case,
			\item $G^+ G^- P^+ \subset \tau P^+ \cup G^- P^+$ since \[g_+ h_- P^+ = (g h_l^{-1})_+ (h_l)_+ h_- (h_r)_+ P^+ \subset G^+ \tau P^+\]
			whenever $h_- \neq 1$.
		\end{itemize}
		For left multiplications with $\tau$, we have that $\tau G^- P^+ = G^+ \tau P^+ \subset \tau P^+ \cup G^- P^+$ and that $\tau^2 P^+ = P^+$.
		We also know that $g_+ \tau = g_+ (g^{-1})_+ (g_l^{-1})_- (\tau(g_r)^{-1})_+ \in G^- P^+$ whenever $g_+ \neq 1$ so that $P^+ \tau \subset \tau \Inn(G^+,G^-) \cup G^- P^+ \subset \tau P^+ \cup G^- P^+$. This shows that $\tau P^+ \tau \cup G^- P^+ \tau \subset G^- P^+ \cup \tau P^+$.			
	\end{proof}
\end{lemma}

\begin{proposition}
	Let $J$ be a division hermitian cubic norm structure and consider $M = G(J)/P^+$ and define $U_{gP^+} = g G^+ g^{-1}$ which acts on the left on $M$. This is a Moufang set.
	\begin{proof}
		It is obvious that $h U_{g P^+} h^{-1} = U_{hg P^+}$ for each $h \in G(J)$ so the second property for Moufang sets holds. 			
		Moreover, from the previous lemma, we learn that the map $G^- \cup \{ \tau \} \longrightarrow G(J)/P^+$ defined by $g \mapsto g P^+$ is bijective. So, we see that $U_{\tau P^+} = G^-$ acts regularly. Now, this immediately implies that each $U_{g P^+} = (g \tau^{-1}) U_{\tau P^+} (\tau g^{-1})$ acts regularly.
	\end{proof}
\end{proposition}

Now, we give an alternative description of the Moufang set $G^- \cup \{ \infty\}$. In terms of $U_\infty$ and a permutation $\tau$ of $G^- \cup \{ \infty\}$ mapping $\infty$ to $1$.
This automatically defines all root groups, using $U_{g} = g_- \tau U_{\infty} \tau^{-1} g^{-1}_-$ for all $g \in G^-$. 
 This is one of the descriptions one often studies, e.g., \cite[section 3]{Tom06}. The Moufang set structure we obtain is basically the one obtained in \cite[Theorem 5.16]{Tom2019}.

\begin{remark}
	We will use $\tau_{(\sqrt{2},1)}$ which is a grade reversing element that acts as $\exp_-(x,y) \mapsto \exp_+(x,y)$. 
	If $\tau_{(\sqrt{2},1)} \notin G$ we add it and we observe that this is the same as adding a generator $\tau_{(\sqrt{2},1)} \tau_g$ to $\Inn(G^+,G^-) \subset G$ so that $\tilde{G}/\tilde{P}^+$ does not change (using $\tilde{X}$ to denote the group extended by adding the generator). We remark that $\tilde{G}$ acts on $M \cong \tilde{G}/\tilde{P}^+$.
	
	This modification is not necessary over fields of characteristic $2$ since $\tau_{(\sqrt{2},1)} = \tau_{(0,1)}$ in that case. If the characteristic is different from $2$, one can also show that the modification is unnecessary, since $\tau_{(1,1/2)} \tau_{(\sqrt{2},1)}$ acts as $z \mapsto 2^{-i} z$ for $z \in L_i$, which can be shown to be an element of the group.
\end{remark}

\begin{proposition}
	The Moufang set $G(J)/P^+$ is isomorphic to the Moufang set $G^- \cup \infty$ formed from $G$ and $\tau$ the order $2$ automorphism which acts as $\tau(\infty) = 1$, $\tau(g) = g_l^{-1}.$
	\begin{proof}
		Under the previous construction, we can identify $\tau$ with left multiplication with $\tau_{(\sqrt{2},1)}$, which acts on $G/P^+ \cong G^- \cup \infty$. This yields a Moufang set from a permutation in the sense of \cite[Section 3]{Tom06} isomorphic to the Moufang set we started with.
		
		So, we only have to determine what this precise permutation action is.
				
		For $g \in G^-$, we have that \[\tau_{(\sqrt{2},1)} g_-P^+ = \tau_{(\sqrt{2},1)} \tau_{g_r} \tau_{g_r^{-1}} g_- P^+ \overset{(*)}{=} \tau_{(\sqrt{2},1)} \tau_{g_r}(\tau^{-1}_{g_r} \cdot g_l^{-1})_- P^+ \]
		
		which equals \[ \tau_{(\sqrt{2},1)} (g_l^{-1})_+ \tau_{g_r} P^+ = \tau_{(\sqrt{2},1)} (g_l^{-1})_+ \tau_{(-\sqrt{2},1)} P^+ = (g_l^{-1})_- P^+.\]
		To see that $(*)$ holds, expand $\tau_{g_r}^{-1}$. 
		
	\end{proof}
\end{proposition}

\begin{remark}
	Since the automorphism is given by $g \mapsto g_l^{-1}$, it is easy to see our construction as coming from a \textit{division pair} in the sense of \cite{Loos14, Loos15}, since $\exp_\pm(g) \mapsto \exp_\mp(g_l^{-1})$ is of order $2$.
	We remark that this result is very much expected, given \cite[Theorem 5.1.5]{Tom2019}.
	
	This gives a straightforward definition of $\tau(g) = g_l^{-1} = (g^{-1})_r$, using Equation \ref{def gr}. So, $\tau$ sends $((a,v),(u,av + v^\sharp))$ to
	\[  \left(\left(\frac{\bar{u}a - a \bar{a} a - \overline{N(v)}}{\nu(\bar{u},a,v)}, \frac{(\bar{u} - T(v,v)) v - \bar{a}v^\sharp + v \times v^\sharp}{\nu(\bar{u},a,v)}\right), (\gamma, w)\right)\]
	with \[\gamma = \frac{ \bar{u}}{\nu(\bar{u},a,v)} + \frac{2 N(v) \overline{N(v)} + 2 a \bar{a} T(v^\sharp,v^\sharp) - 2 (\bar{u} a N(v) + u \bar{a} \overline{N(v)})}{\nu(\bar{u},a,v)^2}\] and where $w$ is uniquely determined by the element having to lie in the group.
	
	We remark that a comparison between our formula for $\tau$ and the formula for $\tau$ given in \cite[Theorem 3.7.1]{muhlherr2022}, should in principle be possible. The different setups allows us to easily match terms, but making the correspondence precise is a bit harder.	
	A different and easier formula one can compare is given in \cite[Lemma 6.29]{boelaert2013moufang}.
\end{remark}

	\section{Classification of division hermitian cubic norm structures}
	
	Now, we build towards our main result: a classification of division hermitian cubic norm structures. First, we look at the degenerate case in which $N(v) = 0$ for all $v$, which corresponds to groups of type $\prescript{2}{}{A}_n$. If $N$ is not identically $0$, we are working with a hermitian cubic norm structure coming from a semisimple cubic norm structure. Using the classification of the latter class obtained by Racine \cite[Theorem 1]{Racine72}, we can build towards a classification for division hermitian cubic norm structures.
	
	\subsection{The $N = 0$ case}
	
	\begin{lemma}
		Consider a division hermitian cubic norm structure such that $N(v) = 0$ for all $v$.
		Then $v^\sharp = 0$ for all $v$ and $T(v,v) \neq 0$ for all $v \neq 0$.
		\begin{proof}
			For any $v$, we compute \[N(v + v^\sharp) = N(v) + N(v^\sharp) + \overline{N(v)} T(v,v) + T(v^\sharp,v^\sharp).\] So, we conclude that $T(v^\sharp,v^\sharp) = 0$.
			Hence $g = ((0,v^\sharp),(0,0)) \in G$ with $\nu(g) = 0$. This implies that $v^\sharp = 0$.
			
			Note that $\nu((0,v),(u,0)) = u\bar{u}$ for all $v$. Hence $u \neq 0$ when $v \neq 0$. This implies that $T(v,v) \neq 0$, since $((0,v),(0,0))$ is an element of $G$ with if $T(v,v) = 0$.
		\end{proof}
	\end{lemma}

	\begin{lemma}
		Consider a hermitian cubic norm structure $J$ over a quadratic field extension $K = R[t]/(t^2 - t + \alpha)$ such that $N(v) = 0$ and $T(v^\sharp,v^\sharp) = 0$ for all $v$.
		This hermitian cubic norm structure is division if and only if the quadratic form $q(a,v) = a\bar{a} - T(v,v)$ is anisotropic, i.e., $a\bar{a} = T(v,v)$ implies $(0,0) = (a,v)$.
		
		\begin{proof}
			We compute for $g = ((a,v), (T(v,v) + a\bar{a})t + \lambda(t - \bar{t}),  av + v^\sharp)$ that
			\begin{align*}
				\nu(g) = & \; (T(v,v) + a\bar{a})^2 t\bar{t} + \lambda(- T(v,v) - a\bar{a} - \lambda) (t - \bar{t})^2 - a\bar{a} T(v,v)\\
				= & \;  \alpha T(v,v)^2 + (2 \alpha - 1) a\bar{a} T(v,v) + \alpha (a\bar{a})^2 - (1 - 4 \alpha ) \lambda T(v,v) \\ & - (1 - 4\alpha ) \lambda a\bar{a} - (1 - 4\alpha) \lambda^2  \\
				= & \alpha (T(v,v) - a\bar{a})^2 - (1 - 4 \alpha)(a\bar{a} T(v,v) + \lambda^2 + \lambda a \bar{a} + \lambda T(v,v)).
			\end{align*}
		Using $\lambda = - T(v,v)$ shows that $T(v,v) - a\bar{a} \neq 0$ for all $a \in K$ and $v \in J$ is necessary for $\nu$ to be anisotropic.
		This shows that $q(a,v)$ being anisotropic is a necessary condition. 
		
		For the converse, we know that $g = ((a,v), (a \bar{a} t + T(v,v) \bar{t}, av + v^\sharp))$ is one-invertible for all $(a,v) \neq (0,0)$ since $\nu(g) = q(a,v)^2 \alpha$.
		Moreover, for \[g_\mu = g \cdot ((0,0),(\mu(t - \bar{t})))\] we have 
		\[ \nu(g_\mu) = - (1 - 4 \alpha) \mu^2 - (1 - 4 \alpha ) q(a,v) \mu + \alpha q(a,v)^2.\]
		This has a solution $\nu(g_\mu) = 0$ for $\mu = - q(a,v) \xi $ if and only if
		\[ \xi^2 - \xi + \alpha/(4 \alpha - 1) = 0. \]
		Now, if such a $\xi$ exists, $K$ would be isomorphic to $R[t]/(t^2 - t)$ and would not be a field extension by \cite[Lemma 4.1.6]{Michiel2025}.
		\end{proof}
	\end{lemma}

We conclude with the following proposition:

	\begin{proposition}
		Consider a hermitian cubic norm structure $J$ over a quadratic field extension $K/R$ and let $n(a) = a\bar{a}$ for $a \in K$. The following are equivalent:
		\begin{itemize}
			\item $N(v) = 0$ for all $v$ and $J$ is division,
			\item $N(v) = 0, v^\sharp = 0$ for all $v$,  and the map $(a,v) \mapsto n(a) - T(v,v)$ is anisotropic.
		\end{itemize}
		Moreover, any vector space $M$ over a quadratic field extension $K/R$ equipped with a hermitian map $h$ such that $(a,v) \mapsto n(a) - h(v,v)$ is anisotropic induces a division hermitian cubic norm structure with $m^\sharp = 0$ and $N(m) = 0$ for all $m \in M$.
	\end{proposition}

	\begin{remark}
		\label{rem: inv in degen case easy}
		In case that $N(v) = 0$ and $v^\sharp = 0$ for all $v$, we can consider the structurable algebra as part of the class of structurable algebras defined from a hermitian form. Namely, let $A$ be an associative $R$-algebra, $M$ be an $A$-module, and $h : M \times M \longrightarrow A$ hermitian. Then $A \times M$ forms a structurable algebra with operation $(a,m)(b,n) = (ab + h(m,n),\bar{a}n + bm)$, cfr. \cite[section 5.3]{OKP} (to switch from the left action we use here to the right action needed and the left and right hermitian forms use $\bar{a}n = na$ and $h(a,b) = \tilde{h}(b,a)$). We use the action $(k,v) \mapsto \bar{k} v$ to embed our structurable algebra into this class, which is allowed since $K$ is commutative.
		
		Using the representation of the Lie algebra established in \cite[section 5.3]{OKP} it is not too hard to show that $((a,m), (u, \bar{a}m))$ (with $u + \bar{u} = a\bar{a} + h(m,m)$) is one-invertible if and only if $u - h(m,m)$ is invertible (under the assumption that $w - \bar{w}$ is invertible for some $w$). We explain this in quite a bit more detail in Appendix \ref{app: B}.
		Now, applying this with $m = 0$ shows that $A$ is a division algebra if all elements (except $(0,0) \in G$) are one-invertible.
		We know that $u - h(m,m) + \overline{u - h(m,m)} = a\bar{a} - h(m,m)$ so that we obtain the same quadratic form as before. If $a\bar{a} = h(m,m)$, one can choose $u = h(m,m)$ and obtain an element in the group that is not one-invertible.
	\end{remark}

	\subsection{The $N \neq 0$ case}

	\begin{lemma}
		Consider a division hermitian cubic norm structure $J$ over $K/R$ such that $N(v) \neq 0$ for some $v$.
		The associated split hermitian cubic norm structure can be obtained from a semisimple cubic norm structure.
		\begin{proof}
			By taking the quadratic extension $L$ such that $L \otimes K \cong L[s]/(s^2 - s)$, one sees that $N(sv) = s N(v) \neq 0$ while $g_{s,v} = ((0,sv),(0,\bar{s} v^\sharp))$ is not invertible since $\nu(g_{s,v}) = 0$. Now, \cite[Theorem 4.1.7, Lemma 1.1.10]{Michiel2025} shows that $sJ$ (and $\bar{s}J$) forms a cubic norm structure $J'$. Moreover, $Q_{g_{s,v}}$ is an isotopy $\bar{s}J \longrightarrow sJ$.
			
			The semisimplicity follows from \cite[34.23 and 39.1]{Skip2024}, given that $T(sv,\cdot) = T(sv^\sharp,\cdot) = 0$ and $N(sv) = 0$ cannot happen all at the same time, and that $L$ is a field.
		\end{proof}
	\end{lemma}

	\begin{remark}
		\label{rmk: semisimple cns classification}
		Given the relation with semisimple cubic norm structures $B$, it is useful to note that these are classified \cite[Theorem 39.6]{Skip2024}. Namely, either the norm is anisotropic, $B \cong H_3(C,\Gamma)$ (the hermitian $3 \times 3$-matrices over a composition algebra $C$ with respect to a certain involution), or $J$ is not simple and is as in \cite[Exercise 34.24]{Skip2024}. Moreover, over an algebraically closed field, the only division hermitian cubic norm structure is the field itself with norm $x \mapsto x^3$.
	\end{remark}

	We write $Q(K, \alpha)$ for a quaternion algebra. Here, $K$ is a quadratic étale extension of $R$ and $\alpha$ is an element of $R$.
	This quaternion algebra is an associative algebra $K \oplus Kj$ with subalgebra $K$, a generator $j$ such that $j^2 = \alpha$, and $k j = j \bar{k}$ for $k \in K$. This algebra comes with a standard involution $k + k'j \mapsto \bar{k}- k'j$. We remark that any quaternion algebra is of this form.

	\begin{proposition}
		
		Consider a division hermitian cubic norm structure $J$ over $K/R$ such that $N(v) \neq 0$ for some $v$ for which the associated cubic norm structure is not simple. 		
		Then, there exists a $w \in J$ such that
		\begin{itemize}
			\item $w^\sharp = 0, T(w,w) \neq 0$, and hence the quaternion algebra $Q = Q(K, T(w,w))$ is contained in the structurable algebra (using a non-standard involution)
			\item $W^\perp = \{ v \in J | T(v,w) = 0 \}$ is a $Q$-module.
			\item $(v,v') \mapsto h(v,v') = T(v',v) + v \times v' \in Q$ for all $v,v'\in W^\perp$ is $Q$-hermitian.
			\item $K \times J \cong Q \times W^\perp$ as structurable algebras, where the latter algebra has $(q,v) (q', v') = (qq' + h(v,v'), \bar{q} v' + q'v)$ as multiplication. 
		\end{itemize}
		Moreover, consider any quaternion algebra $Q = Q(K,\alpha)$, the generator $w$ such that $w^2 = \alpha$, and the unique involution  of $Q$ such that $w \mapsto w$ and $k \mapsto \bar{k}$ for $k \in K$.
		Let $M$ be a $Q$-module endowed with an anti-quadratic map $q : M \longrightarrow Kw$, satisfying $q(w \cdot m) = \alpha q(m)$, and let $T$ be the unique map $T : M \times M \longrightarrow K$ satisfying $T(a, w \cdot b)w  = q(a,b)  \alpha$ for all $a,b \in M$. Then $Q \times M$ defines a structurable algebra of a division hermitian cubic norm structure if and only if $(q,v) \mapsto q\bar{q} - T(v,v) - q(v,v)$ is anisotropic. 
		Furthermore, any division hermitian cubic norm structure satisfying the conditions of this proposition is of this form.
		\begin{proof}
			Because of \cite[Theorem 39.6]{Skip2024}, we know that the associated cubic norm structure is of the form $F \times J(M,q,e)$. To be precise, if we take the quadratic extension $L$ such that $K \otimes L \cong L[s]/(s^2 - s)$, we have that $(sJ, \bar{s} J)$ is a cubic norm pair by \cite[Theorem 4.1.7]{Michiel2025}. Now, $N(sv) = s N(v) \neq 0$ for some $v$. Note that $s N(v) = s l$ for a unique $l \in L$. Hence, we can apply \cite[Lemma 1.1.10]{Michiel2025} to obtain a cubic norm structure (over $L$) with operations $(N',\sharp', T')$ such that $N'(sx) sl = s N(x)$, $l (sx)^{\sharp'} = Q_{sv} (sx)^\sharp$, $s l^2 T'(sx,sy) = T((sx),Q_{\bar{s}v^\sharp} (sy))$ where $N', \sharp',$ and $T'$ coincide with the operators of \cite[34.24]{Skip2024}.

			First, it is useful to establish some additional relations between the operations on the hermitian cubic norm structure and the associated cubic norm structure. We have
			\begin{equation}
				l (Q_{sv} \bar{s} x)^{\sharp'} = Q_{sv} (Q_{sv} \bar{s}x)^\sharp = Q_{sv} Q_{\bar{s} v^\sharp} sx^\sharp = N(sv)^2 x^\sharp = l^2 s x^\sharp
			\end{equation} and
			\begin{equation}
				T'(sx, Q_{sv} y) = T(sx,\bar{s} y).
			\end{equation}

			We want to show that the $w$ with $Q_{sv} w = 1 \in F \subset F \times J(M,q,e)$ satisfies the desired properties.
			First and foremost, we observe that $(Q_{sv} w)^{\sharp'} = 0$, which implies $w^\sharp = 0$, and that $0 = N(Q_{sv} w) = s^2 N(v)^2 \overline{N(w)}$, which shows $N(w) = 0$. Given that we work in a division hermitian cubic norm structure, this immediately implies that $T(w,w) = \alpha \neq 0$. We immediately see that $K \oplus Kw$ is a quaternion algebra with involution $(k,k'w) \mapsto (\bar{k}, k'w)$.
		
			Secondly, any $y$ such that $T(y,w) = 0$ satisfies $T'(sy, Q_{sv} w) = 0$. This guarantees that the element $Q_{sv}y^\sharp = l(sy)^{\sharp'} \in \langle Q_{sv}w \rangle$ and thus $y^\sharp \in \langle w \rangle$.

			Thirdly, we verify whether $W^\perp$ is a $Q$-module. The only possible action with the required isomorphism of structurable algebras is $(k,k'w)u = \bar{k}u + \bar{k'} (w \times u)$ for $(k,k'w) \in K \oplus Kw$ and $u \in W^\perp$.
			This is an action if $w \times (w \times u) = T(w,w) u$ for all $u$, since
			\[ (l,l'w)((k,k'w) u) = (l, l'w)(\bar{k}u + \bar{k'} w \times u) = \overline{lk} u + (\overline{lk'} + k\bar{l}') w \times u + k'\bar{l'} w \times (w \times u)\]
			while
			\[  (lk + l'\bar{k'}T(w,w), (lk' + \bar{k}l')w)u = \overline{(lk + l'\bar{k'}T(w,w)}u + (\overline{lk'} + k\bar{l'}) w \times u. \]
			To verify that $w \times (w \times u) = T(w,w) u$, we first take the anti-quadratic form $f$ on $W^\perp$ such that $u^\sharp = f(u) w$. For any $u$ in $W^\perp$, we know that
			\[ (f(u) w) \times (w \times u) = u^\sharp \times (u \times w) = N(u) w + T(w,u^\sharp) u = \overline{f(u)} T(w,w) u,\]
			so that the desired equality holds whenever $f(u)$ is invertible. Applying \cite[Proposition 12.24]{Skip2024} yields the desired result, since $f(u) = 0$ for all $u$ implies that $N(u + kw) = 0$ for all $u$ and $k$ and thus $N(v) = 0$ for all $v$.
			
			The final property we have to verify is that $h$ is hermitian. It is obviously $K$-antilinear in the second argument and $\overline{h(a,b)} = h(b,a)$.
			So, we only have to check whether $h(w \times a,b) = w h(a,b)$. 
			We have $T(v', v \times w) = T(w, v \times v')$ which equals $w \cdot (v \times v')$. To verify that $(w \times v) \times v' =w T(v',v) = T(v,v') w$, we use that $(w \times v) \times v'$ is necessarily of the form $\lambda w$ for some $\lambda \in K$ and obtain $\lambda T(w,w) = T((w \times v) \times v', w) = T((w \times v) \times w, v') = T(w,w) T(v,v')$.
			 
			 For the moreover-part, compare with Remark \ref{rem: inv in degen case easy} to see that the described algebra is division. This is spelled out in more detail in Appendix \ref{B3}.
		\end{proof}
	\end{proposition}

	\begin{remark}
		If the $Q(K,\alpha)$-module is trivial, we have $N(v) = 0$ for all $v$. Hence, we place the structurable algebras of the previous theorem for which the module is trivial in the class of hermitian cubic norm structures coming from a hermitian form $T$.
	\end{remark}

		\begin{theorem}
			\label{thm: possible hcns}
			Suppose that $J$ is a division hermitian cubic norm structure over $K$. Define the norm $n(k) = k\bar{k}$ on $K$. Now, one of the following holds:
			\begin{itemize}
				\item $J^\sharp = 0$, in which case it is obtained from a hermitian form $T$ such that $(a,v) \mapsto n(a) - T(v,v)$ is anisotropic,
				\item the associated cubic norm structure is semisimple, in which case either
				\begin{itemize}
					\item the associated cubic norm structure is not simple and $J \cong K w \oplus M$ with $M$ a non-trivial $Q(K,\alpha)$-module ($Q(K,\alpha) = K \oplus K w$ with $w^2 = \alpha$) endowed with an anti-quadratic map $M \longrightarrow Kw$ satisfying $q(w \cdot m) = \alpha q(m)$, such that the map \[((k, k'w), m) \mapsto k\bar{k}+ k'\bar{k'} - T(q(m, w \cdot m),w)/ \alpha - q(m,m)\] is anisotropic,	
					\item the associated cubic norm structure is simple and:
					\begin{itemize}			
					\item either, $N$ is anisotropic and $ \dim_K J \ge 2$ and $J$ does not correspond to a purely inseparable field extension of degree $3$ over a field of characteristic $3$,
					\item or the associated cubic norm structure is isomorphic to $H_3(C,\Gamma)$ for a composition algebra $C$ of dimension $>1$.
				\end{itemize}
				\end{itemize}
			\end{itemize}			
		\end{theorem}
	
	The theorem stated follows obviously from the preceding propositions and lemmas, except the additionally imposed restrictions, i.e., $\dim_K J \neq 1$, $J$ is not a purely inseperable field extension, and the associated norm structure not being isomorphic to $H_3(C,\Gamma)$ for a $1$-dimensional composition algebra $C$. Therefore, we shall prove this theorem after establishing the link with algebraic groups that allows us to explain why these restrictions apply. We remark that we only excluded finite dimensional cubic norm structures.

	We shall compare the projective elementary group $G(J) = \langle G^+, G^- \rangle$, or more precisely its Zariski-closure (the minimal closed subgroup functor of $\Aut(L)$ containing $G^+$ and $G^-$), to specific algebraic groups for finite dimensional hermitian cubic norm structures $J$. In what follows, we shall also write $G(J)$ for the Zariski-closure. This group is closed, smooth, and connected \cite[Ex
	position VIB, Proposition 7.1,(i) and Corollary 7.2.1]{demazure1970proprietes}.

	\begin{lemma}
		\label{lem: rk 1}
		The element $\psi(\lambda)$, for invertible $\lambda$, that acts on the Lie algebra as $z \mapsto \lambda^i z$ whenever $z \in L_i$, is contained in $G(J)$. This $\psi$ defines a maximal split torus of $G(J)$ whenever $J$ is a division hermitian cubic norm structure.
		\begin{proof}
			The construction of $\psi(\lambda)$ was performed in Proposition \ref{prop: g one inv iff nu inv with ass} using the $\beta(\lambda,\mu)$.
			Suppose that $\phi$ defines a morphism of the multiplicative group into $G(J)$ and that $\phi(\mu)$ commutes with $\psi(\lambda)$ for all $\lambda$ and $\mu$.
			This defines another $\mathbb{Z}$-grading on our Lie algebra and both gradings are compatible, i.e., define a $\mathbb{Z}^2$-grading. Assume that $L_2$ is $i$-graded in the new grading induced by $\phi$. Now, take nonzero $z \in L_1$ which is $j \neq i/2$-graded. We can therefore find $\tilde{z} \in G^+$ such that $\Psi(\lambda) \tilde{z} \Psi(\lambda)^{-1} = (\lambda^j \cdot \tilde{z})$; since each $g \in G^+$ with $g_1 = \ad \; z$ can be used to compute $\tilde{z}$ using $\psi(\lambda ) g \psi(\lambda^{-1}) = (\lambda^j \cdot \tilde{z}) ((0,0),(\lambda^i s, 0))$.
			
			Now, we see that \[\lambda^{4j} \nu(\tilde{z}) = \nu(\lambda^j \cdot \tilde{z}) = \nu(\psi(\lambda)
			\tilde{z} \psi(\lambda^{-1}))) = \lambda^{2i} \nu(\tilde{z}) \]
			using that $\nu$ is a quartic norm and using Lemma \ref{lem: def nu}, combined with the fact that $L_{-2}$ is $-i$-graded, to compute $\nu(\lambda^j \cdot \tilde{z})$. So, $L_1$ is $j/2$-graded.
			Arguing analogously one can show that $L_{-1}$ is $-j/2$-graded. This shows that $\phi(\lambda) = \psi(\lambda)^{j/2}$. Hence, $\psi$ defines a maximal split torus.
		\end{proof}
	\end{lemma}

	Now, we will link these hermitian cubic norm structures to groups with a certain Tits index. The possible Tits indices and the meaning of these indices are listed in explained in \cite{TitsIndex}.

	\begin{lemma}
		Suppose that $J$ over $K/R$ is a finite dimensional division hermitian cubic norm structure and $J^\sharp = 0$. We have that $G(J)$ is an adjoint linear algebraic group of type $\prescript{2}{}{A}_{n + 2,1}$ with $n = \dim_K J$.
		\begin{proof}
			By extending to the algebraic closure, it is not hard to see that one obtains a group of type $A_{n+1}$. Lemma \ref{lem: rk 1} shows that we obtain a rank one group. It is also not hard to see that the Galois involution of $K/R$ acts non-trivially on the roots.
			
			Namely, over $\Phi[s]/(s^2 - s)$ with $\Phi$ algebraically closed. One can realize $G^+$ as the upper diagonal block-matrices in $\mathbf{Sl}_{n+2}(\bar{\Phi})$ (the blocks correspond to acting on $\bar{\Phi} \oplus \bar{\Phi}^n \oplus \bar{\Phi})$). One can also realize $G^+$ certain upper diagonal block-matrices over $\mathbf{Sl}_{n+2}(\bar{\Phi}[s]/(s^2 - s))$ that interact in a certain way with the involution $s \mapsto 1 -s$ (in line with the construction employed in \cite[Lemma 3.30]{OKP}), which makes it easy to observe the action of the Galois involution on the roots.   
		\end{proof}
	\end{lemma}

	\begin{lemma}
		Suppose that $J$ over $K/R$  is a hermitian cubic norm structure for which the associated cubic norm structure is semisimple but not simple, i.e., is a finite dimensional division hermitian cubic norm structure coming from a quaternion algebra and a non-trivial module for that algebra. The group $G(J)$ is an adjoint linear algebraic group of type $\prescript{1,2}{}{D}_{n + 3,1}$ with $2n + 1= \dim_K J$ and $n \ge 1$.
		\begin{proof}
			The associated cubic norm structure (over the algebraic closure) is described in remark \cite[Remark 2.5.23]{muhlherr2022} (the $B$ case cannot happen since $\dim_K(J)$ is necessarily odd) and one obtains $D_{n+3}$. The relative rank is necessarily $1$.
			
			To exclude $\prescript{3,6}{}{D}_{4,1}$ note that the quadratic extension that splits the hermitian cubic norm structure yields a group of type $\prescript{3,6}{}{D}_{4,2}$ (as $2$ and $3$ are prime), which implies that the associated cubic norm structure is anisotropic.
		\end{proof}
	\end{lemma}

	\begin{lemma}
		Suppose that $J$ over $K/R$ is a division hermitian cubic norm structure with $N$ not identically zero, with associated cubic norm structure of the form $H_3(C,\Gamma)$. Then $G(J)$ is an adjoint linear algebraic group of type $\prescript{2}{}{E_{6,1}^{35}}$, $\prescript{}{}{E}_{7,1}^{66}$, or $\prescript{}{}{E}_{8,1}^{133}$.
		\begin{proof}
			The associated cubic norm structure is of the form $H_3(C,\Gamma)$. Over algebraically closed fields, $G(J)$ is the split adjoint group of type $F_4$, $E_6$, $E_7$ or $E_8$, depending on the dimension of $C$. So, $G(J)$ is a rank one algebraic group of type $F_4$, $E_6$, $E_7$ or $E_8$ corresponding to whether $C$ has dimension one, two, four, or eight. The anisotropic kernels of these groups (not over the algebraic closure), are groups of type $C_3$, $A_5$, $D_6$ and $E_7$.
			Now, there cannot be any division hermitian cubic norm structure of type $F_4$, since any rank one group of type $F_4$ has $B_3$ as anisotropic kernel.
			
			To observe the aforementioned types are correct, note that over the algebraic closure the anisotropic kernel becomes isomorphic the group of \cite[Theorem 6.4]{Loos_1979} for the associated cubic norm pair, using \cite[Corollary 2.3.7]{Michiel2025}. The dimensions of associated the Lie algebra uniquely determine the type of the anisotropic kernel, except in the $C_3$ case. Using the functorial construction of \cite{Jacobson68} of the Lie algebra (which can be carried over to the group), one can verify that $C_3$ is the correct type. Now, one can now recognize $G(J)$ as the Chevalley group of the correct type (acting on the Chevalley Lie algebra modulo its center), the root system is obtained by adding the usual $\mathbb{Z}$-grading to the grading obtained from a maximal torus of the (now split) anisotropic kernel\footnote{The link between these cubic norm structures and groups of type $F_4, E_6, E_7,$ and $E_8$ is well-known. See, for example, \cite[Theorem 2.5.22]{muhlherr2022}. The argumentation here, with the anisotropic kernel does 2 things: showing that we have the right anisotropic kernel and illustrating that one obtains the full Chevalley group over the algebraic closure.}.
		\end{proof}
	\end{lemma}

	\begin{lemma}
		Suppose that $J$ over $K/R$ is a division hermitian cubic norm structure for which the associated cubic norm structure is simple. Over the algebraic closure, it is either the split hermitian cubic norm structure obtained from $H_3(C,\Gamma)$, in which case $G(J)$ has type $\prescript{2}{}{E_{6,1}^{35}}$, $\prescript{}{}{E}_{7,1}^{66}$, or $\prescript{}{}{E}_{8,1}^{133}$, or it has dimension $3$ and $G(J) \cong \prescript{3,6}{}{D}_{4,1}$.
		\begin{proof}
			Over the algebraic closure, the hermitian cubic norm structure is either of the first type and we an apply the analysis of the previous lemma.
			If it is not, we can apply \cite[Theorem 46.8]{Skip2024}, to see that the associated cubic norm structure can, firstly, be a form of the Freudenthal algebra $F \oplus F \oplus F$ for a field $F$, which corresponds to groups of type $D_4$. The division requirement yields a group of type $\prescript{3,6}{}{D}_{4,1}$. To exclude $\prescript{1,2}{}{D}_{4,1}$, note that the quadratic extension that splits the hermitian cubic norm structure creates a group of relative rank at least $2$. Using this for $\prescript{1,2}{}{D}_{4,1}$, one can show that $N$ is isotropic and $J \ncong H_3(C,\Gamma)$. Hence, $J$ comes from the construction involving a quaternion algebra and is not simple.

			If the cubic norm structure is not of these forms, \cite[Theorem 46.8]{Skip2024} says that the cubic norm structure is a field $F$ with norm $N(x) = x^3$ over $F$ or over a subfield $G$ over which $F$ is purely inseperable. However, these possibilities cannot occur.
			Namely, for $F$ over $F$ we obtain a group of type $G_2$ which has either rank $2$ or $0$. In the case of $F$ over $G$ we obtain a subgroup of $G_2(F)$ with a two dimensional $G$-torus.
		\end{proof}
	\end{lemma}

	\begin{proof}[Proof of Theorem \ref{thm: possible hcns}]
		\label{prf: thm possible hcns}
		In the previous lemma we excluded the division hermitian cubic norm structures we excluded in the theorem. In the lemma before that we excluded $H_3(F,\Gamma)$ over a field $F$ as a possible associated hermitian cubic norm structure for composition algebras $F$ of degree $1$. This was everything we had to exclude.
	\end{proof}	
	
		\begin{remark}
			All infinite dimensional hermitian cubic norm structures can be thought of as being contained in the classes $\prescript{2}{}{A}_{n + 2,1}$, $\prescript{1}{}{D}_{n,1}$, and $\prescript{2}{}{D}_{n,1}$, since they all come from certain anisotropic hermitian forms $J \times J \longrightarrow K$ or $M \times M \longrightarrow Q$.
		\end{remark}
	
		\begin{remark}
			Whether the projective elementary group associated to the hermitian cubic norm structure is precisely equal to the associated algebraic group $\mathbf{G}$ associated to it, evaluated in $R$, is related to the Kneser-Tits conjecture. A review paper on this conjecture is \cite{Gille09}. One can restate the problem as in \cite{Alsaody21TW}: the conjucture asks whether for a simply connected isotropic $\mathbf{G}$ over $R$ the abstract group $\mathbf{G}(R)$ coincides with the subgroup generated by the unipotent radicals of the minimal parabolic $R$-subgroups of $G$. For a lot of groups, Tits proved the affirmative \cite{Tits78}. However, the conjecture does not hold for all groups. Over global fields, one can use \cite[Théorème 8.3]{Gille09} to see that it holds. In \cite[Theorem 8.1]{Alsaody21TW}, the authors, for example, show that groups of type $E^{78}_{7,1}$ satisfy the conjecture.
		\end{remark}
	
		\subsection{The precise link with rank one algebraic groups}
		
		Now, we want to show that each group with one of the Tits indices that one could obtain, comes from an hermitian cubic norm structure. To achieve this, we want to use \cite[Theorem 4.2.5]{Michiel2025}, which speaks of operator Kantor pairs. So, first we shall establish how one obtains an operator Kantor pair from such a group. We will not state a definition, however we shall identify the underlying algebraic data one needs to define an operator Kantor pair in Remark \ref{rem: def okp}.
		Then, we shall look at isotopes of the structurable algebra and corresponding operator Kantor pairs. Thereafter, we are ready for the final result.
		
		\begin{lemma}
			\label{lem: descr okp}
			Suppose that $G$ is a simple linear algebraic group with one of the aforementioned Tits indices. Consider opposite parabolics and their unipotent radicals $(U^+,U^-)$. These unipotent radicals form an operator Kantor pair.
			\begin{proof}
				Consider the action of $G$ on $\Lie(G)$. Now, one can apply \cite[Theorem 3.27]{OKP} to obtain the desired result. So, we will explain why the conditions to apply that theorem hold.
				
				The opposite parabolics correspond to the grading induced by the rank one torus contained in them. With respect to this grading, one has $\Lie(U^\pm) = U^\pm_1 \oplus U^\pm_2$ where $U^\pm_i$ has weight $\pm i$. This makes it easy to see them as proper vector groups.
				The conditions on $\exp(o_{i,j}(x,y))$ are easily proven using the grading induced by the rank one torus contained in the opposite parabolics. The other conditions follow easily from the fact that $U^\pm_2$ is one dimensional.
			\end{proof}
		\end{lemma}
	
		\begin{remark}
			\label{rem: def okp}
			An operator Kantor pair consists of certain group functors $U^+$ and $U^-$ with certain maps $o_{i,j} : U^\pm \times U^\mp \longrightarrow U^{\pm}$ for $i > j$. The relevant maps that play a role in the definition of an operator Kantor pair \cite[Definition 3.21]{OKP}, are $Q^\text{grp} = o_{2,1}$, $T = o_{3,1}$ and $P = o_{3,2}$. We also know that $T$ maps to $[U^\pm,U^\pm]$ and for $P$ only the projection onto $U^\pm/[U^\pm,U^\pm]$ is required.
			The $Q$ and $T$ operators are the same as the operators we used to define the actions of $G^+$ and $G^-$ on the Lie algebra. The operator $P$ is actually fully redundant.
		\end{remark}
	
		\begin{definition}
			Suppose that $(G^+,G^-,Q,T,P)$ is an operator Kantor pair such that $f : G^+ \longrightarrow G^-$ is an isomorphism and $f O_x fy = O_{f x} y$ for all $x, y \in G^+$ and all $O \in \{Q, T, P\}$, then we call $(G^+, \tilde{Q}, \tilde{T}, \tilde{P})$, with $\tilde{O}_x y = O_x f y$ for $O \in \{Q, T, P\}$, an \textit{operator Kantor system}. This is equivalent to definition \cite[Definition 3.29]{OKP}, using that $f$ can always be thought of as the identity if we reparametrize $G^-$ using $g \in G^+$ to denote $fg \in G^-$. The condition on $f$ is necessary and sufficient for having shared operators $Q, T$ and $P$.
			
			The \textit{double} of an operator Kantor system is the operator Kantor pair \[(G^+,G^-,Q,T,P).\]
		\end{definition}
	
		\begin{definition}
			Let $(G^+,G^-)$ and $(H^+,H^-)$ be operator Kantor pairs and consider $(f,g) : (G^+,G^-) \longrightarrow (H^+,H^-)$ a pair of group homomorphisms. This pair is called a \textit{homomorphism} if it preserves all the operators.
		\end{definition}

		We now will look at isotopes of the structurable algebra associated to a hermitian cubic norm structure. Consider the skew-dimension 1 structurable algebra $A = K \oplus J$ coming from a hermitian cubic norm structure $J$ over $K/R$ and consider an element $a \in A$ with $\nu((\sqrt{2}a, a \bar{a}))$ invertible. We shall establish that one can construct (under a mild condition on $a$) a new hermitian cubic norm structure $J^{\langle a \rangle}$ over $L/R$ for some $L$ such that $a$ plays the role of the unit in $L \oplus J^{\langle a \rangle}$. 
		
		The operator Kantor pair for the newly obtained structurable algebra can be described as a double $(G^+, \tau_{(\sqrt{2}a, a \bar{a})} G^+)$ of $(G^+, Q \tau_{(\sqrt{2}a, a \bar{a})}, T \tau_{(\sqrt{2}a, a \bar{a})}, P \tau_{(\sqrt{2}a, a \bar{a})})$. This is how one typically constructs isotopes for Jordan pairs \cite{loos1975jordan} and it is more or less clear from \cite[Equation 1.13]{Allison84} that this is how usual isotope for structurable algebras behaves.
	
		\begin{lemma}
			\label{lem: isotope}
			Consider a hermitian cubic norm structure $(J,N,\sharp,T)$ over $K = R[t]/(t^2 - t + \alpha$).
			Suppose that $a$ is an element of the associated structurable algebra for which the scalar $\nu((\sqrt{2}a, a\bar{a}))$ is invertible and for which there exists $(a,u) \in G^+$ such that $u - \bar{u} = \lambda(t - \bar{t})$ with $\lambda$ invertible.
			Then there exists another hermitian cubic norm structure $(J',N',\sharp',T')$ over $L/R$ for some $L$ defining an isomorphic operator Kantor pair for which the associated structurable algebra is isotopic in which the element $a$ plays the role of the unit.
			
			Moreover, the element $(a,u) \in G^+$ always exists whenever $4 = 0$ or $1/2 \in R$.
			\begin{proof}
				One can verify that $\nu(\sqrt{2}a, a \bar{a}))$ coincides with the $\nu$ of \cite[(1.3), or easier Proposition 1.10]{Allison84} and the definition of the other $\nu$ does not contain the scalar $\sqrt{2}$. We use $\tilde{a}$ for the unique element for which $\tau_{(\sqrt{2}a, a \bar{a})} (\tilde{a}) = a$ and we will show that $\tilde{a}$ can play the role of the unit (which works since $a$ and $\tilde{a}$ play interchangeable roles since $(\sqrt{2}a, a\bar{a})_r = (\sqrt{2}a, a\bar{a})_l = (\sqrt{2}\tilde{a}, \tilde{a}\bar{\tilde{a}}))$.
				
				We work with the $1$-generated hermitian cubic norm structure of Proposition \ref{prop: universal model}. 
				We extend our scalars once again to $\mathbb{Z}[z_1,z_2, \nu((\sqrt{2}a, a \bar{a}))^{-1}]$ with $a = (z_1 t + z_2 \bar{t}, v)$.
				
				Since the underlying module is a free $\mathbb{Z}$-module, we can extend scalars to $\mathbb{Q}$. Using the notion of an isotope \cite[section 7]{AllisonHein81} and that $a$ is conjugate invertible with conjugate inverse $2\tilde{a}$, we obtain a hermitian cubic norm structure over $\mathbb{Q}$ in which $\tilde{a}$ has the role of $1$. Moreover, the map $P_u$ of \cite{AllisonHein81} (it is easiest to see using the $P$ of \cite[Theorem 13]{All99} coincides with the relevant restriction\footnote{If there seems to be (an implicit) additional scalar $1/2$ in \cite[Theorem 13]{All99}, recall that the Kantor pair associated to a structurable algebra is rescaled by a factor $2$ over there.} of $\tau_{(\sqrt{2}u, u \bar{u})}$). We can use the obvious automorphism $(G^+,G^-) \cong (G^+, \tau_{(\sqrt{2}a, a \bar{a})} G^+)$ to see $(G^+,G^-)$ as the double of the operator Kantor system $(G,Q^{\langle a \rangle}, T^{\langle a \rangle}, P^{\langle a \rangle})$ using $O^{\langle a \rangle}_{x} y = O_x \tau_{(\sqrt{2}a, a\bar{a})}$, for operators $O$.
				For $\tau_{(\sqrt{2}a, a \bar{a})} y \in \tau_{(\sqrt{2}a, a \bar{a})} G^+$ and $x \in G^+$, we have
				\[ O_{\tau_{(\sqrt{2}a, a \bar{a})} y} x =  \tau_{(\sqrt{2}a, a \bar{a})} \tau_{(\sqrt{2}\tilde{a}, \tilde{a} \bar{\tilde{a}})} O_{\tau_{(\sqrt{2}a, a \bar{a})} y} x = \tau_{(\sqrt{2}a, a \bar{a})} O_{y} \tau_{(\sqrt{2}a, a \bar{a})} x\]
				for all operators $O$.
				
				If $\mathbb{Q}$ was faithfully flat over $\mathbb{Z}$ (which it is not), we could apply \cite[Theorem 4.25]{Michiel2025}, to prove that the operator Kantor pair obtained from $(G,Q^{\langle a \rangle}, T^{\langle a \rangle}, P^{\langle a \rangle})$ can be obtained from a hermitian cubic norm structure with $\tilde{a}$ playing the role of the unit. We can use the argumentation of first part of the theorem to obtain the desired quadratic extension $L/R$, then extend scalars and thereafter apply \cite[Theorem 3.2.7]{Michiel2025} to obtain the same result.
				
				Now, consider subgroup $E$ of $G^+$ formed by the elements with first coordinate in the span of $Q^{<a>}_{(\tilde{a}, u)} \tilde{a} = Q_{(\tilde{a}, u)} a$ for all possible $u$. We note that the module of possible first coordinates $F$ for $E$ is a free rank $2$ module generated by elements of the form $Q^{\langle a\rangle }_{(\tilde{a},\bar{u})} \tilde{a} = \tilde{t}$ using $F = \langle \tilde{t}, Q^{\langle a \rangle}_{(0,t - \bar{t})} \tilde{a} \rangle$. If one extends scalars to $\mathbb{Q}$ one obtains a quadratic extension of $\mathbb{Q}[z_1,z_2,x,y,\alpha, N_1, N_2, (1 - 4 \alpha)^{-1}, \nu(\sqrt{2}a, a\bar{a})^{-1}]$ by $\tilde{t}^2 - \tilde{t} = q$ for some $q$ in the base ring.
				One can verify that $Q^{\langle a \rangle }_{(\tilde{a},u)} \tilde{t} = - q$ using that $Q_{(1,t)} 1 = \bar{t}$ and $Q_{(1,t)} t = t - t^2$ in any operator Kantor pair obtained from a hermitian cubic norm structure, to conclude that $q \in R$. In particular, we see that $F$ can be given the structure of a quadratic extension over $\mathbb{Z}$ if we can find $\tilde{t}$ such that $F = \langle \tilde{a}, \tilde{t} \rangle$ and $1 + 4q$ is invertible.
				We remark that $[\tilde{t},\bar{\tilde{t}}]_2$ is uniquely determined mod $2$ and it equals \[\tilde{t} + \bar{\tilde{t}} = \tilde{a}\bar{\tilde{a}}  \mod 2,\]
				while 
				\[ (\tilde{a}\bar{\tilde{a}})^2 \equiv \nu(\sqrt{2}\tilde{a}, \tilde{a}\bar{\tilde{a}}) \mod 4.\]
				
				This shows that $[\tilde{t},\bar{\tilde{t}}] = d (t - \bar{t})$ for the standard free generator $t - \bar{t}$ of $G^+_2$ and a $d \in R$ such that $d^2 \equiv \nu(\sqrt{2}\tilde{a},\tilde{a}\bar{\tilde{a}}) \mod 4$.
				Since $(1 + 4q) x = (\tilde{t} - \bar{\tilde{t}})^2 x = [s_2, [s_{-2}, x]]$ for $s_2 = [\tilde{t}, \bar{\tilde{t}}]$, $s_{-2}$ the unique element for which $\tau_{(\sqrt{2}\tilde{a}, \tilde{a}\bar{\tilde{a}})} s_{-2} = s_2$, and arbitrary $x \in L(G^+,G^-)_1,$ we obtain that $1 + 4q = d^2 \nu(\sqrt{2}a, a\bar{a})$.
				So that being able to choose $d$ to be invertible, guarantees that $F = \langle \tilde{a}, \tilde{t} \rangle$ and that $1 + 4q$ is invertible.
				
				The assumption on $(a,u)$ allows us to extend with the scalar $d^{-1}$ and conclude that $F$ is a quadratic extension. To obtain the moreover-case, if one extends with $\mathbb{Z}/4\mathbb{Z}$, we automatically have that $d^{-2} = \nu(\sqrt{2}a, a\bar{a})$ and that $u = \tilde{a} \bar{\tilde{a}}/2 + (t - \bar{t})/2$ gives an invertible $d$ if we extend with $1/2$, since $[\tilde{t}, \tilde{\bar{t}}] = - (t - \bar{t})$.
				
				Extending scalars using a quadratic extension $L$ such that $F \otimes L \cong L[u]/(u^2 - u)$, we obtain an operator Kantor subpair $(E^+_L, \tau_{(\sqrt{2}a, a\bar{a})} E^+_L)$
					 of the form \cite[Example 3.2.4]{Michiel2025}. Now, \cite[Theorem 3.2.7]{Michiel2025} gives us a cubic norm pair over $L$ in which $\tilde{a}$ plays the role of $1$ in the structurable algebra.
				To obtain the hermitian cubic norm pair, we have to look at the split hermitian cubic norm structure over $L$ induced by the cubic norm pair \cite[Theorem 4.1.7]{Michiel2025} and look at the fixpoints under the involution of $L$. On one hand we obtain the operator Kantor pair we started from, on the other hand we obtain a hermitian cubic norm structure with $\tilde{a}$ as unit from which we can recover the operator Kantor pair.
			\end{proof}
		\end{lemma}
	
		\begin{remark}
			The lemma above is the first major victory for the definition of structurable operator Kantor pairs given in \cite{OKP}. Namely, consider an operator Kantor pair coming from a hermitian cubic norm structure. We have established that any element that can play the role of the unit element (i.e., any element that the operator Kantor pair structurable), can be thought of as a unit of an isotope of the corresponding structurable algebra.
		\end{remark}
	
		\begin{theorem}
			Any adjoint simple linear algebraic group over a field with as Tits index one of
			\begin{itemize}
				\item $\prescript{2}{}{A}_{n+2,1}$,
				\item $\prescript{1,2}{}{D}_{n+3,1}$ with $n \ge 1$,
				\item $\prescript{3,6}{}{D}_{4,1}$,
				\item $\prescript{2}{}{E}^{35}_6$,
				\item $E^{66}_{7,1}$,
				\item $E^{133}_{8,1}$,
			\end{itemize}
			comes from a division hermitian cubic norm structure, i.e., it is the closure of the projective elementary group associated to the hermitian cubic norm structure, and each finite dimensional division hermitian cubic norm structure defines a group with Tits index in this list.
			\begin{proof}
				If one extends scalars to the algebraic closure $\Phi$, one finds an adjoint simple linear algebraic group that, in each of the cases, can be obtained from a unique hermitian cubic norm structure compatible with the grading induced by the torus of the rank $1$ group. Namely, in the first case one has $\Phi^n \otimes \Phi[s]/(s^2 - s)$ endowed with the standard hermitian form to $\Phi[s]/(s^2 - s)$ and $N = 0$, $\sharp = 0$. In the second case, one has the structurable algebra $M_2(\Phi) \oplus \Phi^{2n}$ interpreted in the usual way as coming from a hermitian cubic norm structure (detailed formulas for $T$, $\sharp$, and $N$ coincide with the formulas given in \ref{B3}). In the other cases, the hermitian cubic norm structure corresponds to a split hermitian cubic norm structure coming from a Freudenthal algebra.
				
				Now, we want to apply \cite[Theorem 4.25]{Michiel2025} on the operator Kantor pair described in Lemma \ref{lem: descr okp}, combined with the previous lemma, to obtain that the linear algebraic group comes from a hermitian cubic norm structure. We first determine the quartic norm and prove that it is anisotropic to obtain the possible unit from the previous lemma, which will allow us to conclude what we want.				
				
				Any rank one simple algebraic group $G$ defines a Moufang set on the minimal parabolic subgroups with as root groups the corresponding unipotent radicals. 
				We remark that for all of the groups listed here, the center $C_P$ of the unipotent radical of $P$ is one dimensional and we will recover $\nu$ (up to a scalar) from the actions on these centers.
				The grading defines two opposite parabolics $P^+$ and $P^-$ with unipotent radicals $U^+$ and $U^-$.
				For each $x \in U^+$, we have that $x P^- x^{-1} \neq P^-$ so that there exists a unique $y \in U^-$ with $y x P^- x^{-1} y^{-1} = P^+$. In particular, we conclude that $y x C_{P^-} x^{-1} y^{-1} = C_{P^+}$.
				This shows that $\nu_\Phi(x) \neq 0$ for all $x \neq 1$ in $U^+$. In particular, if we choose $e = (\sqrt{2} x ,x \bar{x})$, which exists in $U(R[\sqrt{2}]) \subset U(\Phi)$ we note that $\nu(e)$ is invertible (if $1/2 \in R$ replace $e$ by $e' = (x,x\bar{x}/2)$ and divide $\nu$ by $4$, if the characteristic is $2$ look at $(0,x\bar{x}) = (0,x^{[2]})$ and obtain the same $\nu$), hence we can apply the previous lemma and \cite[Theorem 4.25]{Michiel2025} to obtain a hermitian cubic norm structure in which $x$ plays the role of the unit of the structurable algebra. It is automatically division, since $\nu_R$ coincides with $\nu_\Phi$ up to a non-zero scalar.
				
				To observe that the closure of the group generated by $G^+$ and $G^-$ is $G$, note that it coincides with the group generated by all the unipotent radicals of parabolics of $G$, which is a closed normal subgroup of $G$.
			\end{proof}
		\end{theorem}
	
		\begin{remark}
			The argumentation used above works broader and can probably be extended to prove that any group with Tits index in which one of the circled vertices induces a grading from which one can recover a hermitian cubic norm structure in the split case, comes from a hermitian cubic norm structure. A subtlety one has to keep in mind, is that one has to work with a Tits set as defined in \cite{muhlherr2022} instead of a Moufang set. 
		\end{remark}

	\appendix
	\section{The computer model}
	\label{app A}
	
	\subsection{Implementation}
	
	One can use basic computer algebra software to do most of the desired computations in our one generator model.
	First and foremost, we can embed $K/R$ into $\mathbb{C}/\mathbb{R}$ using any transcendental $\alpha$ with $1 - 4\alpha < 0$, where $\alpha, x, y, N_1, N_2$ are algebraically independent over $\mathbb{Q}$ in $\mathbb{R}$. The involution of $K$ is then induced by the involution of $\mathbb{C}$ fixing $\mathbb{R}$, since $(t - \bar{t})^2 = 1 - 4 \alpha < 0$ so that $t - \bar{t}$ is imaginary.
	 We can also just write $N \in \mathbb{C}$ for $N_1 t + N_2 \bar{t}$.
		
	One defines explicitly 
	\[ (a v + b v^\sharp + c v \times v^\sharp)^\sharp = f_1(a,b,c) v + f_2(a,b,c) v^\sharp + f_3(a,b,c) v \times v^\sharp\]
	with $f_i$ the quadratic map obtained by the matrices
	\[ M_1 = \begin{pmatrix}
		 & & \\ & N & \\ \bar{N} & y & \bar{N} x
	\end{pmatrix}, \quad M_2 =  \begin{pmatrix}
	1 & & \\ & & \\ x & N & y
\end{pmatrix}, \quad \text{and}, \quad M_3 = \begin{pmatrix}
& & \\ 1 & & \\ & & - \bar{N}
\end{pmatrix},\]
using
\[ f_i(a,b,c) = \begin{pmatrix}
	\bar{a} & \bar{b} & \bar{c}
\end{pmatrix} M_i \begin{pmatrix}
\bar{a} \\ \bar{b} \\ \bar{c}
\end{pmatrix}.\]
The hermitian form $T$ can also be expressed using a matrix, using
\[  \begin{pmatrix}
	a & b & c
\end{pmatrix} \begin{pmatrix}
	x & 3N & 2y \\ 3\bar{N} & y & 2 \bar{N}x \\ 2y & 2 N x & xy + 3N\bar{N}
\end{pmatrix} \begin{pmatrix}
\bar{d} \\ \bar{e} \\ \bar{f} 
\end{pmatrix}\]
for $T(av + bv^\sharp + c(v \times v^\sharp),dv + e v^\sharp + f(v \times v^\sharp))$.

The easiest definition of the norm $n$ (we write $n$ for the norm instead of $N$ to avoid confusion with the $N$ that will appear in the definition of the norm) is 
\[ n(w) = T(w,w^\sharp)/3,\]
which yields
\[ b^3\bar{N}^2 + b^2(ax + 2cy)\bar{N} + (-Nc^3 + bc^2x + 3abc)N\bar{N} + (acx + c^2y + a^2)(Ncx + Na + by)\]
for $w = av + bv^\sharp + c (v \times v^\sharp)$.

\subsection{Verified equations}

The actual implementation and a PDF-version of it can be found on \url{https://github.com/MichielSmet/hermitian_cubic_norm}. We used a maple worksheet. We refer to equations that were verified in the maple worksheet by their number.

\begin{lemma}
	\label{lem: app1}
	Set $J$ to be the $3$-dimensional $\mathbb{C}$-vectorspace spanned by $v , v^\sharp$ and $v \times v^\sharp$. The quadruple $(J,n,\sharp,T)$ forms a hermitian cubic norm structure over $\mathbb{C}$.
	\begin{proof}
		For the first axiom $T(a,b^\sharp) = n^{(1,2)}(a,b)$, we use that \[n^{(1,2)}(a,b) = T(a,b^\sharp)/3 + T(b, a \times b)/3\] for all $a, b$. Computer verification (see equation 1.8) shows that \[T(b, a \times b) = 2 T(a ,b^\sharp)\] for $a = a_1 v + a_2 v^\sharp + a_3 v \times v^\sharp$ and $b = b_1 v + b_2 v^\sharp + b_3 v \times v^\sharp$ over $\mathbb{C}[a_1,\dots,b_3]$. So, we conclude that the first axiom holds for arbitrary $a, b$ over arbitrary scalar extensions.
		
		The second axiom $N(c) c = (c^\sharp)^\sharp$ is easily verified for $c = c_1 v + c_2 v^\sharp + c_3 v \times v^\sharp$ over $\mathbb{C}[c_1,c_2,c_3]$ using a computer (see equation 1.9), which implies that it holds for all $c$ over all scalar extensions.
		The last two axioms follow from Lemma \ref{lem: redundant axs}.
	\end{proof}
\end{lemma}

\begin{lemma}
	\label{lem: app2}
	The $J$ over $K/R$ of Proposition \ref{prop: universal model} is a hermitian cubic norm structure.
	\begin{proof}
		Lemma \ref{lem: app1} shows that the axioms hold for $a = a_1 v + a_2 v^\sharp + a_3 v \times v^\sharp$ and $b = v_1 v + b_2 v^\sharp + b_3 v \times v^\sharp$ over $K[a_1,\dots,b_3]$. Hence, we conclude that it holds for arbitrary elements $a$ and $b$ over arbitrary scalar extension.			
	\end{proof}
\end{lemma}

Now, we want to list the equations related to Lemma \ref{lem: universal model one invertible} that we verified using a computer. These equations must hold for the universal one-generated hermitian cubic norm structure $J$ over $K/R$ introduced in Proposition \ref{prop: universal model}.
In this lemma, we consider an element $g = ((a,v), (u, av + v^\sharp)) \in G^+$.
From this element, we defined
\begin{itemize}
	\item $\alpha_1 = - ua + a^2 \bar{a} + \overline{N(v)}$, which is an element of $K$,
	\item $\alpha_2 = (-u + T(v,v))v + \bar{a}v\sharp - v \times v^\sharp$, which is an element of $J$,
	\item $\gamma = \nu(u,a,v) u + 2 N(v) \overline{N(v)} + 2a\bar{a} T(v^\sharp,v^\sharp) - 2(ua N(v) + \bar{u}\bar{a} \overline{N(v)})$ in $K$,
	\item $w = \nu(u,a,v)^2 \alpha_1 \alpha_2 + \alpha_2^\sharp$, in $J$,
	\item $\nu(a,u,v) \cdot g_r = ((\alpha_1,\alpha_2), (\gamma, w)) \in G^-,$
\end{itemize}
where $\lambda \cdot (x,y) = (\lambda x, \lambda^2 y)$. We recall that 
\[ Q_{(x,u)} y = (x \bar{y}) x - uy.\]

\begin{lemma}
	\label{lem: app eqs}
	The following equations, with $a \in K$, $v \in J$, hold in the hermitian cubic norm structure $J$ over $K/R$ defined in Proposition \ref{prop: universal model} 
	\begin{enumerate}		
		\item $\gamma + \bar{\gamma} = \alpha_1 \overline{\alpha_1} + T(\alpha_2,\alpha_2)$,
		\item $\gamma - \bar{\gamma} = \nu(u,a,v) (u - \bar{u})$,
		\item $P_g (0,s) = s ( \bar{u}a - a^2 \bar{a} - \overline{N(v)}, (\bar{u} - T(v,v))v - \bar{a} v^\sharp + v \times v^\sharp)$,
		\item $ - (a,v) + Q_g g_r + P_g(0, \gamma - \overline{\gamma}) = 0$,
		\item $ 2 \nu(u,a,v)V_{g,g_r} + V_{g, (u - \bar{u}) (a,v)} + (4\nu(a,u,v) + L_{(u  - \bar{u})}^2) = 0$,
		\item $ 2 \nu(u,a,v) V_{gr, g} + V_{(u - \bar{u}) (a,v), g} + (4\nu(a,u,v) + L_{(u  - \bar{u})}^2) = 0$,
		\item $ 6 (u - \bar{u}) (\alpha_1,\alpha_2) + 3 (u - \bar{u})^2 (a,v) - V_{g, (u - \bar{u}) (a,v)} (a,v) = 0.$
	\end{enumerate}
\begin{proof}
	These equations were verified using a computer. We list the equation (2.x) of the maple worksheet for each of the equations listed here.
	Namely, (2) follows immediately from the definition of $\gamma$, (1) corresponds to (2.11), (3) to (2.8) using that $3 x_3 = 3 x_2 x_1 - x_1^3$ for $\exp(g) = 1 + x_1 + x_2 + x_3 + x_4$ with $x_i$ acting as $+i$ on the degree, (4) is (2.8) up to a factor $\nu(u,a,v)$, (5) is  (2.14), (6) is (2.15), and (7) is (2.16).
\end{proof}
\end{lemma}

\begin{lemma}
	\label{lem: app main}
	Lemma \ref{lem: universal model one invertible} holds.
	\begin{proof}
		We prove this lemma using the equations of Lemma \ref{lem: app eqs}.
		Equation (1) shows that $g_r \in G^-$. Equation (2) shows that \cite[Theorem 13]{All99} should be applied to show that $(a,B) = ((\alpha_1,\alpha_2), - \nu(u,a,v) L_{u - \bar{u}}/2)$ equals $((a,v), - L_{u - \bar{u}}/2)^\wedge$ (using the notation of that theorem). The additional minus sign in front of $L_{u - \bar{u}}/2$, comes from the different sign conventions we use, given that $Q_{(0,u - \bar{u})} x = [(u - \bar{u})_2, x_{-1}] = - ((u - \bar{u}) x)_1 = - (L_{(u - \bar{u})} x)_1$

		Equations (5) and (6) are the first two equations of \cite[Theorem 13]{All99}, rescaled by a factor $2 \nu(u,a,v)$. Equation (7) is the third equation of \cite[Theorem 13]{All99}, rescaled by a factor $-6 \nu(u,a,v)$.
		
		The fourth equation of the theorem applies automatically since $L_{u - \bar{u}}^2 = \lambda \text{Id}$ for some $\lambda$ so that the equation follows from $0 = K_{x,x} = [x,x]$.
		
		So, \cite[Theorem 13]{All99} implies that $g$ is one-invertible with $g_r$ as described. Subtleties involving the base ring do not matter since $g_l = (g^{-1})_r$ (while $\nu(g^{-1}) = \nu(g)$) and
		\[ \tau = g_l g g_r \]
		reverses the grading of the Lie algebra over $\mathbb{C}$, which contains our Lie algebra over $\mathbb{Z}$. 
	\end{proof}
\end{lemma}

\begin{remark}		
	Equations (3) and (4) of Lemma \ref{lem: app eqs}, relate to a different way to determine $g_r$, which always works if $1/2$ is contained in the base ring. Namely, if $g$ is one-invertible, $g_r$ should be the unique element such that $\nu(u,a,v)(\gamma - \bar{\gamma}) = \nu(u,a,v)^2 (u - \bar{u})$, $T_g g_r = 0$, and equation (4) holds.
\end{remark}

\section{Structurable algebras coming from hermitian forms}
\label{app: B}

\subsection{Basic definitions}

In this section, we explain how one-invertibility works in structurable algebras coming from hermitian forms.
Such a structurable algebra is constructed from an associative algebra $C$ with involution $\bar{\cdot}$, a right-$C$-module $M$ and a right-hermitian form $h : M \times M \longrightarrow C$. The structurable algebra $A$ is formed by elements of $C \times M$, equipped with the multiplication
\[ (a,m)(b,n) = (ab + h(n,m), na + m\bar{b}),\]
and involution
\[ \overline{(a,m)} =(\bar{a},m).\]
One can define an operator Kantor pair corresponding to these algebras, as done in \cite[section 5.3]{OKP}.
In this appendix, we explain how one-invertibility works in this setup, to expand upon remark \ref{rem: inv in degen case easy}.

We work with the associative algebra described in \cite[section 5.3]{OKP} consisting of by $3 \times 3$-matrices lying in
\[ \mathcal{A} = \begin{pmatrix}
	C & A & C \\ A & \mathcal{E} & A \\ C & A & C
\end{pmatrix},\]
using a specific space $\mathcal{E}$ of $A^{2 \times 1} \oplus A^{1 \times 2}$ endomorphisms.
For the definition of the product, see the reference.
This algebra is $\mathbb{Z}$-graded if one looks at the weights of the conjugation action of 
\[ \begin{pmatrix}
	t \\ & 1 \\ & & t^{-1}
\end{pmatrix}.\]

We have $A = C \oplus M$ with hermitian forms $h : M \times M \longrightarrow C$, $h^\pm : A \times A \longrightarrow C : (a + m, b + n) \mapsto a\bar{b} \pm h(m,n)$.
We write $c \cdot_\epsilon (a,m) = (ca, - cm)$ for $c \in C$, $(a,m) \in C \times M$.
We also write $(a,m) \cdot_\epsilon c = (\bar{c}a, - \bar{c}m)$ for the corresponding right action.

We have groups $G^+$ and $G^-$ isomorphic to $\{ (a,c) \in A \times C | c + \bar{c} = h^-(a,a) \}$, that are realized as
\[ (a,c) \mapsto \begin{pmatrix}
	1 & a & - c \\ & 1 & - a  \\ & & 1
\end{pmatrix}\]
and 
\[ (a,c) \mapsto \begin{pmatrix}
	1 & & \\ - a & 1 & \\ - c & a & 1
\end{pmatrix}.\]
This realization is a reparametrization of the one employed in \cite{OKP}.

\subsection{Criterion for one-invertibility}

\begin{lemma}
	Consider $(a,c) \in G^+$ and $(b,d)$ in $G^-$. Then $U = (a,c)(b,d)$ has entries $U_{ij} = 0$ for all $i + j < 4$ and all $i$ if and only if $d = \bar{c}^{-1}$ and $b = c^{-1} \cdot_\epsilon a$.
	Moreover, in the case that $c$ is invertible, there exists such a $(b,d) \in G^-$ for which $U_{ij} = 0$ with $i + j \ge 4$.
	\begin{proof}
		We have $U_{11} = 1 - h^+(a,b) + cd$, $U_{12} = a - c \cdot_\epsilon b$, and $U_{21} = -b + a \cdot_\epsilon d$.
		We compute $0 = h^+(U_{12}, b) = h^+(a,b) - c h^-(b,b)$ so that $h^+(a,b) = c(d + \bar{d})$.
		If we substitute this value for $h^+(a,b)$ in $U_{11}$, we obtain $1 = c\bar{d}$.		
		The formula for $b$ is immediately obtained from $U_{12}$.
		
		Now, suppose that $c$ is invertible and put $d = \bar{c}^{-1}$ and $b = c^{-1} \cdot_\epsilon a$.
		Note that $h^{-}(b,b) = c^{-1} h^{-}(a,a) \bar{c}^{-1} = \bar{c}^{-1} + c^{-1}$, so that $(b,d) \in G^-$.
		We also see that
		\[ U_{11} = 1 -  h^-(a,a) \bar{c}^{-1}+ c\bar{c}^{-1}  = 0.\]
		We automatically have $U_{12} = 0$ and $U_{21} = 0$ if one uses the definition of $a \cdot_\epsilon d$.
	\end{proof}
\end{lemma}

We call $g \in G^+$ \textit{one-invertible} if there exists $g_r, g_l \in G^-$ such that 
\[ g_l g g_r\] reverses the grading of the Lie subalgebra generated by the Lie algebras of $\mathcal{A}$ generated by the Lie algebras of $G^+$ and $G^-$.

\begin{lemma}
	\label{lem: app B criterium}
	Consider $(a,c) \in G^+$ and suppose that there exists an element $w \in C$ such that $w - \bar{w}$ is invertible, then $(a,c)$ is one-invertible if and only if $c$ is invertible.
	\begin{proof}
		The previous lemma gives an explicit formula for $g_r$ in the case that $c$ is invertible.
		Arguing analogously, one obtains $g_l = ( \bar{c}^{-1} \cdot_\epsilon a, \bar{c}^{-1})$ and $U = g_l g g_r$ has only non-zero entries $U_{ij}$ with $i + j = 4$. Hence, $g_l g g_r$ reverses the $\mathbb{Z}$-grading of the full algebra $\mathcal{A}$ and in particular the grading of the Lie algebra generated by the Lie algebras of $G^+$ and $G^-$.
		
		To observe that $c$ is necessarily invertible, we use that the action of any grade reversing $g_l gg_r$ with $g = (a,c) \in G^+$ is given by 
		\[ v \mapsto c v \bar{c} ,\]
		for $v \in \Lie(G^-)_{-2}$. Using that there exists a $v$ such that $w - \bar{w} = cv\bar{c}$ shows that $c$ must be invertible.
	\end{proof}
\end{lemma}

\begin{remark}
	The existence of an element $w$ such that $w - \bar{w}$ is invertible is a technical and unnecessary condition. However, the proof of the previous lemma becomes a lot easier and this condition is always satisfied in the cases under consideration in the rest of the article.
	For the general case, one can show that
	\[\Phi(t) = \begin{pmatrix}
		t^2 \\ & 1 \\ & & t^{-2}
	\end{pmatrix}\]
	lies in the group generated by $G^+(R[\sqrt{2}])$ and $G^-(R[\sqrt{2}])$ for all invertible $t$. Therefore, any grade reversing element $\tau$ satisfies $\Phi(t) \tau = \tau \Phi(t^{-1})$. This shows that any grade reversing element for the Lie algebra will reverse the grading of $\mathcal{A}$, which immediately forces the $c$ in $(a,c)$ to be invertible using the first lemma.
	
\end{remark}

\subsection{Application to hermitian cubic norm structures}

\begin{definition}
	We call a structurable algebra arising from a hermitian form \textit{division} if $c$ is invertible for each $(a,c) \in G$ with $a \neq 0$. By the previous lemma, this is often the same as asking that each element is one-invertible. 
\end{definition}

There are two structurable algebras from hermitian forms that played a role in the rest of the article. Both classes of algebras contain an element $w$ for which $w - \bar{w}$ is invertible. So, Lemma \ref{lem: app B criterium} can be applied.

\subsubsection{Division hermitian cubic norm structures with $N = 0$}

\label{B2}
Recall that we have a quadratic extension $K/R$ and a $K$-module $M$ equipped with a hermitian form $h : M \times M \longrightarrow K$. As for hermitian cubic norm structures, it is not hard to show that $K$ is necessarily a field if the structurable algebra is division. Moreover, each $((k,m),k') \in G$ with $(k,m) \neq 0$ has to satisfy $c \neq 0$. Since $c + \bar{c} = k \bar{k}- h(m,m)$, it is sufficient and necessary that $(k,m) \mapsto k \bar{k} - h(m,m)$ is anisotropic.

In this case, we can obtain a hermitian cubic norm structure on $M$ by simply setting $N = 0$, $\sharp = 0$ and $T = h$.

\subsubsection{Division hermitian cubic norm structures with $N \neq 0$ coming from quaternionic hermitian forms}

\label{B3}

Recall that we have a quaternion algebra $Q(K,\alpha) = K \oplus Kw$ with $w^2 = \alpha$ and a $Q(K,\alpha)$ module $M$ with a special kind of hermitian form. Namely, we have anti-quadratic $q : M \longrightarrow Kw$ such that $q(w \cdot m) = \alpha q(m)$ and use $T(a, w \cdot b) w = q(a,b) \alpha$ to define $T : M \times M \longrightarrow K$. Combined, these define $h = T + q(\cdot,\cdot): M \times M \longrightarrow Q(K,\alpha)$. The involution on $Q(K,\alpha)$ is given by $k + k'w \mapsto \bar{k} + k'w$.

If the associated structurable algebra is division, we immediately see that $Q(K,\alpha)$ must be division and that $(q,m) \mapsto q\bar{q} - T(m,m) - q(m,m)$ must be anisotropic, using Lemma \ref{lem: app B criterium}.

To describe the associated hermitian cubic norm structure, use
\begin{itemize}
	\item $J = Kw \oplus M$,
	\item $T(k_1 w + m_1, k_2 w + m_2) = k_1 w \overline{k_2 w} + T(m_1,m_2) = k_1 w k_2 w + T(m_1,m_2)$,
	\item $(kw + m)^\sharp = \bar{k}(wm) + q(m)$,
	\item $N(kw + m) = k w q(m)$.
\end{itemize}
This is a hermitian cubic norm structure since the equations
\begin{enumerate}
	\item $T(k_1 w + m_1, \bar{k}_2(wm_2) + q(m_2)) = k_1 w q(m_2) + (q(m_1,m_2) w) k_2$,
	\item $N(kw + m)(kw + m) = (k q(m) w) w + (k q(m) w) m$,
\end{enumerate}
combined with Lemma \ref{lem: redundant axs} show that all the hermitian cubic norm structure axioms hold.

	\bibliographystyle{alpha}
	\bibliography{bib}

\end{document}